\documentclass[letterpaper,12pt]{article}
\usepackage{amsmath,amssymb,amsfonts,epsf,epsfig,enumerate,placeins}
\newtheorem{theo}{Theorem}[section]
\newtheorem{prop}[theo]{Proposition}
\newtheorem{lemma}[theo]{Lemma}
\newtheorem{defn}[theo]{Definition}

\newtheorem{cor}[theo]{Corollary}

\newenvironment{proof}{\noindent {\sc Proof}.}
                {\phantom{a} \hfill \framebox[2.2mm]{ } \bigskip}

\newtheorem{innercustomthm}{Theorem}
\newenvironment{customthm}[1]
  {\renewcommand\theinnercustomthm{#1}\innercustomthm}
  {\endinnercustomthm}

\newcommand{\ZZ}{\mathbb{Z}}

\def\int{{\rm int}}

\newcommand{\D}{{\mathcal{D}}}

\renewcommand{\P}{{\mathcal{P}}}
\newcommand{\F}{{\mathcal{F}}}
\newcommand{\join}{\bowtie}

\newcommand{\ci}{{\rm Circ}}
\newcommand{\dci}{\overrightarrow{\rm Circ}}
\newcommand{\du}{\, \dot{\cup} \,}

\title{The directed Oberwolfach problem \\ with variable cycle lengths: a recursive construction}
\author{Suzan Kadri and  Mateja \v{S}ajna\footnote{Email: msajna@uottawa.ca. Phone: +1-613-562-5800 ext. 3522. Mailing address: Department of Mathematics and Statistics, University of Ottawa, 150 Louis-Pasteur Private, Ottawa, ON, K1N 6N5,Canada. }  \\ University of Ottawa}

\begin{document}
\maketitle \baselineskip 17pt

\begin{abstract}
The directed Oberwolfach problem OP$^\ast(m_1,\ldots,m_k)$ asks whether the complete symmetric digraph $K_n^\ast$, assuming $n=m_1+\ldots +m_k$, admits a decomposition into spanning subdigraphs, each a disjoint union of $k$ directed cycles of lengths $m_1,\ldots,m_k$. We hereby describe a method for constructing a solution to OP$^\ast(m_1,\ldots,m_k)$ given a solution to OP$^\ast(m_1,\ldots,m_\ell)$, for some $\ell<k$, if certain conditions on $m_1,\ldots,m_k$ are satisfied. This approach enables us to extend a solution for OP$^\ast(m_1,\ldots,m_\ell)$ into a solution for  OP$^\ast(m_1,\ldots,m_\ell,t)$, as well as into a solution for OP$^\ast(m_1,\ldots,m_\ell,2^{\langle t \rangle})$, where $2^{\langle t \rangle}$ denotes $t$ copies of 2, provided $t$ is sufficiently large.

In particular, our recursive construction allows us to effectively address the two-table directed Oberwolfach problem. We show that OP$^\ast(m_1,m_2)$ has a solution for all $2 \le m_1\le m_2$, with a definite exception of $m_1=m_2=3$ and a possible exception in the case that $m_1 \in \{ 4,6 \}$, $m_2$ is even, and $m_1+m_2 \ge 14$.
It has been shown previously that OP$^\ast(m_1,m_2)$  has a solution if $m_1+m_2$ is odd, and that OP$^\ast(m,m)$ has a solution if and only if $m \ne 3$.

In addition to solving many other cases of OP$^\ast$, we show that when $2 \le m_1+\ldots +m_k \le 13$, OP$^\ast(m_1,\ldots,m_k)$ has a solution if and only if $(m_1,\ldots,m_k) \not\in \{ (4),(6),(3,3) \}$.

\end{abstract}

\section{Introduction}

The celebrated Oberwolfach problem (OP), posed by Ringel in 1967, asks whether $n$ participants at a conference can be seated at $k$ round tables of sizes $m_1, \ldots, m_k$ for several nights in row so that each participant sits next to everybody else exactly once. The assumption is that $n$ is odd and $n=m_1+\ldots+m_k$. In graph-theoretic terms, OP$(m_1,  \ldots, m_k)$ asks whether $K_n$ admits a decomposition into 2-factors, each a disjoint union of cycles of lengths $m_1, \ldots, m_k$. When $n$ is even, the complete graph minus a 1-factor, $K_n-I$, is considered instead \cite{HuaKot}. OP has been solved completely in the case that $m_1=\ldots=m_k$ \cite{AlsHag,AlsSch,Hag,HofSch}, and in many other special cases (for example, for $m_1, \ldots, m_k$ all even \cite{BryDan,Hag}, and for $n$ sufficiently large \cite{GloJoo}), but is in general still open. Most pertinent for this paper are the following results on OP.

\begin{theo}{\rm \cite{AdaBry,Dez,FraHol,FraRos,Mes,SalDra}}\label{thm:OPsmall}
Let $3\le m_1 \le \ldots \le m_k$ be  integers with $m_1+\ldots + m_k \le 100$. Then OP$(m_1, \ldots, m_k)$ has a solution if and only if $(m_1, \ldots, m_k) \not\in \{ (3,3),(4,5),(3,3,5),$ $(3,3,3,3) \}$.
\end{theo}

\begin{theo}{\rm \cite{BryDan,Gvo,Hag,Tra}}\label{thm:Tra}
Let $m_1$ and $m_2$ be integers such that $3 \le m_1 \le m_2$. Then $OP(m_1,m_2)$ has a solution if and only if $(m_1,m_2) \not\in \{ (3,3),(4,5) \}$.
\end{theo}

The directed Oberwolfach problem, OP$^\ast(m_1,\ldots,m_k)$, was introduced in \cite{BurSaj}. It asks whether $n$ participants can be seated at $k$ round tables of sizes $m_1,  \ldots, m_k$ (where $n=m_1+\ldots+m_k$)
for several nights in row so that each person sits {\em to the right} of everybody else exactly once. Such a seating is equivalent to a decomposition of $K_n^*$, the complete symmetric digraph of order $n$, into subdigraphs isomorphic to a disjoint union of directed $k$ cycles of lengths $m_1, \ldots, m_k$. The directed Oberwolfach problem with uniform cycle lengths, that is, with $m_1=\ldots=m_k=m$, is denoted OP$^\ast(n;m)$.
The solution to OP$^\ast(n;m)$ has been completed very recently, as seen below.

\begin{theo}{\rm \cite{AdaBry1,BenZha,BerGerSot,BurFraSaj,BurSaj,Lac,Til}}\label{thm:Kn*}
Let $m \ge 2$ and $n \ge 2$. Then OP$^\ast(n;m)$ has a solution if and only if $m|n$ and $(n, m) \not\in\{(4,4),(6,6),(6,3)\}$.
\end{theo}

For $n$ odd, it is easy to see that a solution for OP$(m_1,\ldots,m_k)$ gives rise to a solution for OP$^\ast(m_1,\ldots,m_k)$; we simply take two copies of each 2-factor and direct the two copies of each cycle in two different ways. We thus have the following immediate corollary to Theorems~\ref{thm:OPsmall} and \ref{thm:Tra}.

\begin{cor}{\rm \cite{AdaBry,Dez,FraHol,FraRos,SalDra,Tra}\label{cor:OP}}
Let $3\le m_1 \le \ldots \le m_k$ be  integers such that $n=m_1+\ldots + m_k$ is odd. If $(m_1, \ldots, m_k) \not\in \{ (4,5),(3,3,5) \}$, and $n < 60$ or $k=2$,
then OP$^\ast(m_1,\ldots,m_k)$ has a solution.
\end{cor}

Very little else is known about the non-uniform case; as far as we know, the only other previous results are solutions to OP$^\ast(4,5)$, OP$^\ast(3,3,5)$, and OP$^\ast(2^{\langle b \rangle},3)$ when $2b+3 \equiv 1,3$ or $7 \pmod{8}$ by Shabani and the second author \cite{ShaSaj}. The latter of these results will be subsumed by the much more general results of this paper.

In this paper, we describe a method for constructing a solution to OP$^\ast(m_1,\ldots,m_k)$ given a solution to OP$^\ast(m_1,\ldots,m_\ell)$, for some $\ell<k$, if certain conditions on $m_1,\ldots,m_k$ are satisfied. Such an extension is the simplest in the bipartite case; that is, when all cycle lengths are even.

\begin{theo}\label{thm:main-even}
Let $\ell$ and $k$ be integers such that $1 \le \ell < k$, and let $m_1,\ldots, m_k$ be even positive integers such that  $m_1+\ldots+m_{\ell}=m_{\ell+1}+ \ldots + m_k$.

If OP$^\ast(m_1,\ldots,m_{\ell})$ and OP$^\ast(m_{\ell+1},\ldots,m_k)$ both have solutions, then OP$^\ast(m_1,\ldots,m_k)$ has a solution.
\end{theo}

More sophisticated conditions need to be imposed when the directed 2-factor contains odd cycles or the extension does not exactly double the number of vertices; see Proposition~\ref{prop:main} and Corollary~\ref{cor:main}. This more general approach enables us to extend a solution for OP$^\ast(m_1,\ldots,m_\ell)$ into solutions for OP$^\ast(m_1,\ldots,m_{\ell},t)$ and OP$^\ast(m_1,\ldots,m_{\ell},2^{\langle \frac{t}{2} \rangle})$, as elaborated below.

\begin{theo}\label{the:main-ext}
Let $\ell$ and $m_1,\ldots, m_{\ell}$ be positive integers such that  $m_i \ge 2$ for $i=1,\ldots, \ell$. Let $s=m_1+\ldots+m_{\ell}$ and let $t$ be an integer, $t>s$. Furthermore, let $a$ be the number of odd integers in the multiset $\{ \!\! \{ m_1,\ldots,m_{\ell} \} \!\! \}$, and assume that $a \le 2\lfloor \frac{t}{2} \rfloor -s$.

If OP$^\ast(m_1,\ldots,m_{\ell})$ has a solution, then the following also have a solution:
\begin{enumerate}[(1)]
\item OP$^\ast(m_1,\ldots,m_{\ell},t)$; and
\item if $t$ is even, OP$^\ast(m_1,\ldots,m_{\ell},2^{\langle \frac{t}{2} \rangle})$.
\end{enumerate}
\end{theo}

In particular, this approach allows us to obtain an almost-complete solution to the directed Oberwolfach problem with two tables.

\begin{theo}\label{thm:main2}
Let $m_1$ and $m_2$ be integers such that $2 \le m_1 \le m_2$. Then OP$^\ast(m_1,m_2)$ has a solution if and only if $(m_1,m_2) \ne (3,3)$, with a possible exception in the case that $m_1 \in \{ 4,6 \}$, $m_2$ is even, and $m_1+m_2 \ge 14$.
\end{theo}

This paper is organized as follows. In Section 2 we give the necessary terminology and other prerequisites, and in Section 3 we solve some special small cases of the problem that will be required to fill in the gaps in later, more general constructions. In the next two sections we introduce the main recursive approach to constructing solutions to larger cases of OP$^\ast$ from solutions to smaller cases: for bipartite directed 2-factors (Theorem~\ref{thm:main-even}) in Section 4, and for general directed 2-factors (Proposition~\ref{prop:main} and Corollary~\ref{cor:main}) in Section 5. In Sections 6 and 7, we perform the heavy technical work that allows us to prove our  main extension result, Theorem~\ref{the:main-ext}, in Section 8. This theorem gives rise to many explicit existence results, including Theorem~\ref{thm:main2}, which are also presented in this section. Finally, in Section 9, we completely solve the directed Oberwolfach problem for orders up to 13. Solutions to small cases that do not arise from other results and are not used anywhere else are given in the appendix.

\section{Prerequisites}

We use the symbol $\{ \!\! \{ . \} \!\! \}$ to denote multisets, and $\langle . \rangle$ in the exponent to denote the multiplicity of an element in a multiset or sequence. Thus, for example $\{ \!\! \{ 2,2,2,4 \} \!\! \}=\{ \!\! \{ 2^{\langle 3 \rangle},4^{\langle 1 \rangle} \} \!\! \}=\{ \!\! \{ 2^{\langle 3 \rangle},4 \} \!\! \}$ and $(2,2,2,4)=(2^{\langle 3 \rangle},4^{\langle 1 \rangle})=(2^{\langle 3 \rangle},4)$.

As usual, the vertex set and arc set of a directed graph (shortly {\em digraph}) $D$ will be denoted $V(D)$ and $A(D)$, respectively. All digraphs in this paper are strict, that is, they have no loops and no parallel arcs.

By $K_n$, $\bar{K}_n$,  and $C_m$ we denote the complete graph of order $n$, the empty graph of order $n$,  and the  cycle of length $m$ ($m$-cycle), respectively.
Analogously, by $K_n^*$  and $\vec{C}_m$ we denote the complete symmetric digraph of order $n$ and the directed cycle of length $m$ (directed $m$-cycle), respectively, while the symbol $\vec{P}_m$ will denote a directed path of length $m$; that is, a directed path with $m$ arcs.

A disjoint union of digraphs $D_1$ and $D_2$ is denoted by $D_1 \du D_2$.

A {\em decomposition} of a digraph $D$ is a set $\{ D_1, \ldots, D_k \}$ of subdigraphs of $D$ such that $\{ A(D_1), \ldots, A(D_k) \}$ is a partition of $A(D)$; in this case, we write $D= D_1 \oplus \ldots \oplus D_k$. A {\em $D'$-decomposition} of $D$, where $D'$ is a subdigraph of $D$, is a decomposition into subdigraphs isomorphic to $D'$.

A {\em directed 2-factor} of a digraph $D$ is a spanning subdigraph of $D$ that is a disjoint union of directed cycles. In particular, a
{\em $(\vec{C}_{m_1},\ldots,\vec{C}_{m_k})$-factor} of $D$ is a directed 2-factor of $D$ that is a disjoint union of $k$ directed cycles of lengths $m_1,\ldots,m_k$, respectively.
A {\em $(\vec{C}_{m_1},\ldots,\vec{C}_{m_k})$-factorization} of $D$ is a decomposition of $D$ into $(\vec{C}_{m_1},\ldots,\vec{C}_{m_k})$-factors.

A decomposition, 2-factor, $(C_{m_1},\ldots,C_{m_k})$-factor, and $(C_{m_1},\ldots,C_{m_k})$-fac\-to\-ri\-za\-tion of a graph are defined analogously.

If $D$ is a digraph, and $D_1,\ldots,D_k$ its subdigraphs, then we define a {\em $(D_1,\ldots,D_k)$-subdigraph} of $D$ as a subdigraph of $D$ that is a disjoint union of $k$ digraphs isomorphic to $D_1,\ldots,D_k$, respectively.

The {\em join} of vertex-disjoined digraphs $D_1$ and $D_2$, denoted $D_1 \join D_2$, is the digraph with vertex set $V(D_1 \join D_2)=V(D_1) \du V(D_2)$ and arc set $A(D_1 \join D_2)= A(D_1) \cup A(D_2) \cup \{ (u_1,u_2),(u_2,u_1): u_1 \in V(D_1), u_2 \in V(D_2) \}$.

Let $n$ be a positive integer and $S \subseteq \ZZ_n^\ast$. The {\em directed circulant} of order $n$ with connection set $S$, denoted $\dci(n;S)$, is the digraph with vertex set $V=\{ u_i: i \in \ZZ_n \}$ and arc set $A=\{ (u_i,u_{i+d}): i \in \ZZ_n, d \in S \}$. An arc of the form $(u_i,u_{i+d})$ in $\dci(n;S)$ is said to be of {\em difference} $d$. A subdigraph of $\dci(n;S)$ is called {\em $S$-orthogonal} if it contains exactly one arc of each difference in $S$.
If $S=-S$, then $\ci(n;S)$ will denote the corresponding undirected circulant graph.

The following previous results will be used in our constructions.

\begin{theo}\cite{BerFavMah}\label{the:BerFavMah}
Let $n \in \ZZ^+$, and $a,b \in \ZZ_n^\ast$. If the circulant graph $\ci(n;\{ \pm a,\pm b \})$ is connected and of degree $4$, then it admits a decomposition into two Hamilton cycles.
\end{theo}

\begin{theo}\cite{Wes}\label{the:Wes}
Let $n \in \ZZ^+$ be even, and $a,b,c \in \ZZ_n^\ast$. If $\gcd(n,a,b)\gcd(n,c)=2$ and  the circulant graph $\ci(n;\{ \pm a,\pm b, \pm c \})$ is of degree 6, then it admits a decomposition into three Hamilton cycles.
\end{theo}

\section{1-rotational constructions: special cases}

In this section, we exhibit direct constructions of solutions to some special cases of OP$^\ast$ that will be required in more general results before Section~\ref{sec:small}.
The next definition and lemma provide a framework for a generalized version of the well-known 1-rotational construction.

\begin{defn}\label{def:base=q}{\rm
Let $q$ and $n$ be positive integers such that $q|(n-1)$, and let $K_n^\ast=K_{n-1}^\ast \join K_1$, where $V(K_{n-1})=\{ u_i: i \in \ZZ_{n-1} \}$ and $V(K_1)=\{ u_{\infty} \}$. The {\em base-$q$ differences} of arcs $(u_i,u_j)$, $(u_i,u_{\infty})$, and $(u_{\infty},u_i)$ are defined to be $d_r$, $\infty_r$, and $-\infty_r$, respectively, where $d \in \ZZ_{n-1}$ and $r \in \ZZ_q$ are such that $j-i \equiv d \pmod{n-1}$ and $i \equiv r \pmod{q}$.
When $q=1$, we write simply $d$, $\infty$, and $-\infty$ instead of $d_0$, $\infty_0$, and $-\infty_0$, respectively.}
\end{defn}

\begin{lemma}\label{lem:1-rotational-q}
For $i=1,\ldots,k$, let $m_i \ge 2$ be an integer, and let $n=m_1+\ldots +m_k$. Furthermore, let $q$ be a positive divisor of $n-1$.

If $K_{n-1}^\ast \join K_1$ admits $(\vec{C}_{m_1},\ldots,\vec{C}_{m_k})$-factors
$F_0,F_1,\ldots,F_{q-1}$ that jointly contain exactly one arc of each base-$q$ difference in the set $\{ d_r: d \in \ZZ_{n-1}^\ast \cup \{ \infty, -\infty \}, r \in \ZZ_q \}$, then $K_n^\ast$ admits a $(\vec{C}_{m_1},\ldots,\vec{C}_{m_k})$-factorization.
\end{lemma}

\begin{proof}
With the set-up of Definition~\ref{def:base=q}, let $\rho$ be the permutation $\rho=( u_0 \, u_1 \, \ldots \, u_{n-2}) ( u_{\infty})$.
Since $\rho^q$ preserves base-$q$ differences of the arcs, we can easily see that the $(\vec{C}_{m_1},\ldots,\vec{C}_{m_k})$-factors in
$\{ \rho^{qi}(F_j): i \in\ZZ_{\frac{n-1}{q}}, j \in \ZZ_{q} \}$ jointly contain each arc of $K_n^\ast$ exactly once, and the result follows.
\end{proof}

Note that directed 2-factors $F_0,F_1,\ldots,F_{q-1}$ from Lemma~\ref{lem:1-rotational-q} are often referred to as {\em starter 2-factors} of the $(\vec{C}_{m_1},\ldots,\vec{C}_{m_k})$-factorization.

\begin{lemma}\label{lem:special}
The following problems have a solution:
\begin{itemize}
\item OP$^\ast(4,5)$, OP$^\ast(4,6)$, OP$^\ast(2^{\langle 3 \rangle},4)$, OP$^\ast(4,8)$, and OP$^\ast(2,4,6)$;
\item OP$^\ast(2^{\langle b \rangle},3^{\langle 2 \rangle})$ for $1 \le b \le 5$; and
\item  OP$^\ast(2^{\langle b \rangle},6)$ for $2 \le b \le 4$.
\end{itemize}
\end{lemma}

\begin{proof}
We shall use the set-up from Definition~\ref{def:base=q}. In all cases except for OP$^\ast(4,6)$ and OP$^\ast(4,8)$, we have $q=1$. It thus suffices to exhibit a $(\vec{C}_{m_1},\ldots,\vec{C}_{m_k})$-factor that contains exactly one arc of each difference in $\ZZ_{n-1}^\ast \cup \{ \infty, -\infty \}$.

The following are
\begin{itemize}
\item a $(\vec{C}_{4},\vec{C}_{5})$-factor of $K_9^\ast$:
    $u_0 \, u_1 \, u_5 \, u_3 \, u_0 \; \cup \;
    u_2 \, u_4 \, u_7 \, u_6 \, u_{\infty} u_2$;

\item a $(\vec{C}_{2}^{\langle 3 \rangle},\vec{C}_{4})$-factor of $K_{10}^\ast$:
    $ u_0 \, u_{\infty} \, u_0  \; \cup \; u_1 \, u_3 \, u_1 \; \cup \; u_5 \, u_6 \, u_5 \; \cup \; u_2 \, u_7 \, u_4 \, u_8 \, u_2$;

\item a $(\vec{C}_{2},\vec{C}_{4},\vec{C}_{6})$-factor of $K_{12}^\ast$:
    $ u_0 \, u_{\infty} \, u_0  \; \cup \; u_3 \, u_7 \, u_9 \, u_8 \, u_3 \; \cup \;
    u_1 \, u_2 \, u_5 \; u_{10} \, u_6 \, u_4 \,  u_1$;


\item a $(\vec{C}_{2},\vec{C}_{3}^{\langle 2 \rangle})$-factor of $K_8^\ast$:
    $u_0 \, u_{\infty} \, u_0 \; \cup \; u_1 \, u_2 \, u_4 \, u_1 \; \cup \; u_6 \, u_5 \, u_3 \, u_6$;

\item a $(\vec{C}_{2}^{\langle 2 \rangle},\vec{C}_{3}^{\langle 2 \rangle})$-factor of $K_{10}^\ast$:
    $ u_0 \, u_{\infty} \, u_0   \; \cup \; u_2 \, u_7 \, u_2 \; \cup \; u_1 \, u_3 \, u_4 \, u_1 \; \cup \; u_5 \, u_8 \, u_6 \, u_5$;

\item a $(\vec{C}_{2}^{\langle 3 \rangle}, \vec{C}_{3}^{\langle 2 \rangle})$-factor of $K_{12}^\ast$:
    $ u_0 \, u_{\infty} \, u_0   \; \cup \; u_1 \, u_4 \, u_1 \; \cup \; u_5 \, u_6 \, u_5 \; \cup \; u_3 \, u_{10} \, u_8 \, u_3 \; \cup \; u_2 \, u_7 \, u_9 \, u_2$;

\item a $(\vec{C}_{2}^{\langle 4 \rangle}, \vec{C}_{3}^{\langle 2 \rangle})$-factor of $K_{14}^\ast$: \\
    $ u_0 \, u_{\infty} \, u_0   \; \cup \; u_6 \, u_7 \, u_6 \; \cup \; u_1 \, u_5 \, u_1 \; \cup \; u_2 \, u_{12} \, u_2 \; \cup \; u_3 \, u_{10} \, u_8 \, u_3
    \; \cup \; u_4 \, u_{9} \, u_{11} \, u_4$;

\item a $(\vec{C}_{2}^{\langle 5 \rangle}, \vec{C}_{3}^{\langle 2 \rangle})$-factor of $K_{16}^\ast$: \\
    $ u_0 \, u_{\infty} \, u_0   \; \cup \; u_4 \, u_8 \, u_4 \; \cup \; u_7 \, u_{12} \, u_7 \; \cup \; u_3 \, u_{6} \, u_3
    \; \cup \; u_9 \, u_{11} \, u_9
    \; \cup \; u_1 \, u_{10} \, u_2 \, u_1
    \; \cup \; u_5 \, u_{13} \, u_{14} \, u_5$;


\item a $(\vec{C}_{2}^{\langle 2 \rangle},\vec{C}_{6})$-factor of $K_{10}^\ast$:
    $ u_0 \, u_{\infty} \, u_0  \; \cup \; u_1 \, u_6 \, u_1 \; \cup \; u_2 \, u_3 \, u_5 \; u_8 \, u_7 \, u_4 \,  u_2$;

\item a $(\vec{C}_{2}^{\langle 3 \rangle},\vec{C}_{6})$-factor of $K_{12}^\ast$:
    $ u_0 \, u_{\infty} \, u_0  \; \cup \; u_1 \, u_6 \, u_1 \; \cup \; u_7 \, u_{10} \, u_7 \; \cup \;
    u_2 \, u_3 \, u_5 \; u_9 \, u_8 \, u_4 \,  u_2$;

\item a $(\vec{C}_{2}^{\langle 4 \rangle},\vec{C}_{6})$-factor of $K_{14}^\ast$:  \\
    $ u_0 \, u_{\infty} \, u_0  \; \cup \; u_1 \, u_6 \, u_1 \; \cup \; u_3 \, u_{9} \, u_3 \; \cup \;
    u_{12} \, u_{2} \, u_{12} \; \cup \;
    u_4 \, u_5 \, u_7 \; u_{11} \, u_{10} \, u_8 \,  u_4$,

\end{itemize}
all with the required properties.

For OP$^\ast(4,6)$, we have $n=10$, and we choose $q=3$. By Lemma~\ref{lem:1-rotational-q}, it suffices to find three directed 2-factors that jointly contain exactly one arc of each difference in $\{ d_r: d \in \{ \pm 1, \pm 2, \pm 3, \pm 4, \pm \infty \}, r \in \ZZ_3 \}$. It is not difficult to verify that the following
$(\vec{C}_{4},\vec{C}_{6})$-factors of $K_{10}^\ast$ satisfy the requirement:
    \begin{eqnarray*}
    F_0 &=& u_1 \, u_5 \, u_8 \, u_6 \, u_1 \; \cup \;  u_0 \, u_2 \, u_4 \, u_3 \, u_{\infty} \, u_7 \, u_0, \\
    F_1 &=& u_2 \, u_7 \, u_3 \, u_6 \, u_2 \; \cup \;  u_0 \, u_8 \, u_5 \, u_4 \, u_1 \, u_{\infty} \, u_0, \\
    F_2 &=& u_0 \, u_7 \, u_1 \, u_8 \, u_0 \; \cup \;  u_2 \, u_6 \, u_3 \, u_4 \, u_5 \, u_{\infty} \, u_2.
    \end{eqnarray*}

For OP$^\ast(4,8)$, no 1-rotational solution with $q=1$ exists (see the paragraphs following the proof of Theorem~\ref{thm:small}). The following directed 2-factors form a solution without symmetry:
\begin{eqnarray*}
F_0 &=& u_3 \, u_{10} \, u_4 \, u_{\infty} \, u_3 \; \cup \; u_0 \, u_2 \, u_9 \, u_5 \, u_ 1 \, u_7 \, u_6 \, u_8 \, u_0, \\
F_1 &=& u_1 \, u_4 \, u_7 \, u_3 \, u_1 \; \cup \; u_0 \, u_{\infty} \, u_5 \, u_9 \, u_ 2 \, u_{10} \, u_8 \, u_6 \, u_0, \\
F_2 &=& u_0 \, u_7 \, u_5 \, u_3 \, u_0 \; \cup \; u_1 \, u_2 \, u_{\infty} \, u_9 \, u_{10} \, u_6 \, u_4 \, u_8 \, u_1, \\
F_3 &=& u_3 \, u_8 \, u_4 \, u_6 \, u_3 \; \cup \; u_0 \, u_5 \, u_{10} \, u_2 \, u_ 1 \, u_9 \, u_{\infty} \, u_7 \, u_0, \\
F_4 &=& u_0 \, u_3 \, u_{\infty} \, u_1 \, u_0 \; \cup \; u_2 \, u_6 \, u_9 \, u_8 \, u_ 7 \, u_{10} \, u_5 \, u_4 \, u_2, \\
F_5 &=& u_0 \, u_{10} \, u_3 \, u_9 \, u_0 \; \cup \; u_1 \, u_5 \, u_8 \, u_{\infty} \, u_6 \, u_7 \, u_2 \, u_4 \, u_1, \\
F_6 &=& u_0 \, u_8 \, u_{10} \, u_{\infty} \, u_0 \; \cup \; u_1 \, u_3 \, u_4 \, u_5 \, u_6 \, u_2 \, u_7 \, u_9 \, u_1, \\
F_7 &=& u_3 \, u_5 \, u_7 \, u_8 \, u_3 \; \cup \; u_0 \, u_9 \, u_4 \, u_{10} \, u_ 1 \, u_6 \, u_{\infty} \, u_2 \, u_0, \\
F_8 &=& u_4 \, u_9 \, u_7 \, u_{\infty} \, u_4 \; \cup \; u_0 \, u_1 \, u_8 \, u_5 \, u_ 2 \, u_3 \, u_6 \, u_{10} \, u_0, \\
F_9 &=& u_2 \, u_5 \, u_{\infty} \, u_8 \, u_2 \; \cup \; u_0 \, u_6 \, u_1 \, u_{10} \, u_9 \, u_3 \, u_7 \, u_4 \, u_0, \\
F_{10} &=& u_1 \, u_{\infty} \, u_{10} \, u_7 \, u_1 \; \cup \; u_0 \, u_4 \, u_3 \, u_2 \, u_8 \, u_9 \, u_6 \, u_5 \, u_0.
\end{eqnarray*}
\end{proof}

We remark that a solution to OP$^\ast(4,5)$ (as well as to OP$^\ast(3,3,5)$, given in Appendix~\ref{app1}) has been previously constructed by Shabani and the second author \cite{ShaSaj}, while
a solution to OP$^\ast(4,6)$ of a different form has been obtained by Lacaze-Masmonteil \cite{Lac2}.


\section{A recursive construction of solutions to OP$^\ast$: \\ the bipartite case}

In this section, we begin to describe a method for constructing solutions to larger cases of OP$^\ast$ from solutions to smaller cases. This method is particularly simple and effective when the directed 2-factor is bipartite; that is, when all its cycle lengths are even. We re-state Theorem~\ref{thm:main-even} for convenience.

\begin{customthm}{\ref{thm:main-even}}
Let $\ell$ and $k$ be integers such that $1 \le \ell < k$, and let $m_1,\ldots, m_k$ be even positive integers such that  $m_1+\ldots+m_{\ell}=m_{\ell+1}+ \ldots + m_k$.

If OP$^\ast(m_1,\ldots,m_{\ell})$ and OP$^\ast(m_{\ell+1},\ldots,m_k)$ both have solutions, then OP$^\ast(m_1,\ldots,m_k)$ has a solution,
\end{customthm}

\begin{proof}
Let $n=m_1+\ldots+m_{\ell}=m_{\ell+1}+ \ldots + m_k$.
We need to show that $K_{2n}^\ast$ admits a $(\vec{C}_{m_1},\ldots,\vec{C}_{m_k})$-factorization. Let $D=K_{2n}^\ast= D_1 \join D_2$, where $D_1 \cong D_2 \cong K_n^\ast$, and
$$ V(D_1)=X=\{ x_i: i \in \ZZ_n \} \quad \mbox{ and } \quad V(D_2)=Y=\{ y_i: i \in \ZZ_n \}.$$
Let $\rho$ be the cyclic permutation $\rho=(y_0 \, y_1 \, \ldots \, y_{n-1} )$ on $X \cup Y$ that fixes the vertices in $X$ pointwise.

We first decompose $D= (D_1 \du D_2) \oplus (\bar{D}_1 \join \bar{D}_2)$. By assumption, digraphs $D_1$ and $D_2$ admit a
$(\vec{C}_{m_1},\ldots,\vec{C}_{m_{\ell}})$-factorization and a
$(\vec{C}_{m_{\ell+1}},\ldots,\vec{C}_{m_k})$-factorization, respectively, each with $n-1$ directed 2-factors. Hence $D_1 \du D_2$ admits a $(\vec{C}_{m_{1}},\ldots,\vec{C}_{m_k})$-factorization, and it remains only to construct a $(\vec{C}_{m_{1}},\ldots,\vec{C}_{m_k})$-factorization of $\bar{D_1} \join \bar{D_2}$.

Let $\P_1=\{ X_1,\ldots,X_k \}$ and $\P_2=\{ Y_1,\ldots,Y_k \}$ be partitions of $X$ and $Y$, respectively, such that $|X_i|=|Y_i|=\frac{m_i}{2}$ for $i=1,\ldots,k$. For each $i$, relabel the vertices in $X_i$ and $Y_i$ as $X_i=\{ x_1^{(i)},\ldots, x_{\frac{m_i}{2}}^{(i)} \}$ and $Y_i=\{ y_1^{(i)},\ldots, y_{\frac{m_i}{2}}^{(i)} \}$, and let
$$C^{(i)}= x_1^{(i)} \, y_1^{(i)} \, x_2^{(i)} \, y_2^{(i)} \, \ldots x_{\frac{m_i}{2}}^{(i)} \,  y_{\frac{m_i}{2}}^{(i)} \, x_1^{(i)}.$$
So $C^{(i)}$ is a directed $m_i$-cycle with arcs in $(X \times Y) \cup (Y \times X)$, and $F=C^{(1)} \cup \ldots \cup C^{(k)}$ is a
$(\vec{C}_{m_1},\ldots,\vec{C}_{m_k})$-factor of $\bar{D_1} \join \bar{D_2}$. Since every vertex in $X$ is the tail and the head of exactly one arc of $F$, the directed 2-factors in $\D=\{ \rho^i(F): i \in \ZZ_n \}$ jointly contain each arc in $(X \times Y) \cup (Y \times X)$ exactly once. Hence $\D$ is a $(\vec{C}_{m_{1}},\ldots,\vec{C}_{m_k})$-factorization of $\bar{D_1} \join \bar{D_2}$.
\end{proof}

\begin{cor}\label{cor:even}
Let $a$ and $b$ be positive integers, and $s$ and $t$ be even positive integers such that $as=bt$. Then OP$^\ast( s^{\langle a \rangle}, t^{\langle b \rangle} )$ has a solution.
\end{cor}

\begin{proof}
Assume first that $(a,s),(b,t) \not\in \{ (1,4),(1,6) \}$, so that by Theorem~\ref{thm:Kn*}, both OP$^\ast( as;s)$ and OP$^\ast(bt;t )$ have solutions. Since $as=bt$ and $s,t$ are even, it follows from Theorem~\ref{thm:main-even} that OP$^\ast( s^{\langle a \rangle}, t^{\langle b \rangle} )$ has a solution.

The remaining cases are OP$^\ast(4,4)$ and OP$^\ast(6,6)$, which have solutions by Theorem~\ref{thm:Kn*}, as well as OP$^\ast(2^{\langle 2 \rangle},4)$ and OP$^\ast(2^{\langle 3 \rangle},6)$, which have solutions by Lemmas~\ref{lem:(2,2,4)} and \ref{lem:special}, respectively.
\end{proof}

A repeated application of Theorem~\ref{thm:main-even} and Corollary~\ref{cor:even}, respectively, immediately yields the following two results.

\begin{cor}\label{cor:even2}
Let $M=\{ \!\! \{ m_1,\ldots,m_k \} \!\! \}$ be a multiset of even positive integers with a partition into multisets $P_1,\ldots,P_{\ell}$ with the following properties:
\begin{enumerate}[(i)]
\item $\ell=2^a$ for some $a \in \ZZ^+$;
\item $\sum_{m \in P_i} m=\sum_{m \in P_j} m$ for all $i,j \in \{ 1,2,\ldots,\ell\}$; and
\item for each $i=1,2,\ldots ,\ell$, if $P_i=\{ \!\! \{ m_{1,i},\ldots,m_{k_i,i} \} \!\! \}$, then OP$^\ast(m_{1,i},\ldots,m_{k_i,i})$ has a solution.
\end{enumerate}
Then OP$^\ast(m_1,\ldots,m_k)$ has a solution.
\end{cor}

\begin{cor}\label{cor:even3}
Let $\ell=2^a$ for $a \in \ZZ^+$, and for each $i=1,2,\ldots,\ell$, let $a_i$ and $s_i$ be positive integers with $s_i$ even. If $a_is_i=a_js_j$ for all $i$ and $j$, then OP$^\ast( s_1^{\langle a_1 \rangle}, \ldots, s_{\ell}^{\langle a_{\ell} \rangle} )$ has a solution.
\end{cor}

\section{A recursive construction of solutions to OP$^\ast$: \\ the general case}

In this section, we shall generalize the idea of constructing solutions to larger cases of OP$^\ast$ from smaller ones by allowing cycles of odd length. In this case, the two ``parts'' of the construction will necessarily be of unequal size.

\begin{prop}\label{prop:main}
Let $\ell$, $k$, and $m_1,\ldots, m_k$ be integers such that $1 \le \ell < k$ and $m_i \ge 2$ for $i=1,\ldots, k$. Let $s=m_1+\ldots+m_{\ell}$ and $t=m_{\ell+1}+\ldots+m_{k}$, and assume $s<t$.

Then OP$^\ast(m_1,\ldots,m_k)$ has a solution if the following conditions all hold.
\begin{enumerate}[(1)]
\item OP$^\ast(m_1,\ldots,m_{\ell})$ has a solution.
\item There exist a decomposition $\{D',D''\}$ of $K_t^\ast$ and
 non-negative integers $s_1, \ldots, s_k, t_1, \ldots, t_k$ such that:
   \begin{enumerate}[(a)]
    \item $D'$ admits a $(\vec{C}_{m_{\ell+1}},\ldots,\vec{C}_{m_{k}})$-factorization with exactly $s-1$ directed 2-factors;
    \item $s_1+ \ldots + s_k=s$;
    \item $m_i=2s_i+t_i$ for $i=1,2,\ldots,k$; and
    \item $D''$ admits a decomposition $\D=\{ H_0,\ldots,H_{t-1} \}$ such that for all $j \in \ZZ_t$,
     \begin{itemize}
     \item $V(H_j)=\rho^j(V(H_0))$, where $\rho$ is a cyclic permutation of order $t$ on $V(K_t^\ast)$;
     \item $H_j= D_1^{(j)} \du \ldots \du D_k^{(j)}$, where for each $i=1,\ldots,k$,
         \begin{enumerate}[$\star$]
        \item $D_i^{(j)} \cong \vec{P}_{t_i}$ if $t_i<m_i$, and $D_i^{(j)} \cong \vec{C}_{m_i}$ if $t_i=m_i$;
        \item if $D_i^{(0)}$ is a directed $(x,y)$-path for some vertices $x$ and $y$, then $D_i^{(j)}$ is a directed $(\rho^j(x),\rho^j(y))$-path.
        \end{enumerate}
    \end{itemize}
    \end{enumerate}
\end{enumerate}
\end{prop}

\begin{proof}
Assume Conditions (1) and (2) hold. We need to show that $K_{s+t}^\ast$ admits a $(\vec{C}_{m_1},\ldots,\vec{C}_{m_k})$-factorization. Let $D=K_{s+t}^\ast= K_s^\ast \join K_t^\ast$, where
$$ V(K_s^\ast)=X=\{ x_i: i \in \ZZ_s \} \quad \mbox{ and } \quad V(K_t^\ast)=Y=\{ y_i: i \in \ZZ_t \}.$$
Let $\rho$ be the cyclic permutation $\rho=(y_0 \, y_1 \, \ldots \, y_{t-1} )$ on $X \cup Y$ that fixes the vertices in $X$ pointwise.

Construct subdigraphs $L_1,\ldots,L_k$ of $D$ recursively as follows. Assuming $L_1,\ldots,L_{i-1}$ have already been constructed, we use the subdigraph $D_i^{(0)}$ of $H_0$ to construct $L_i$ as follows.
\begin{enumerate}[(i)]
\item If $t_i=m_i$, then $D_i^{(0)} \cong \vec{C}_{m_i}$, and we let $L_i=D_i^{(0)}$.
\item If $t_i=0$, then $D_i^{(0)} \cong \vec{P}_{0}$. Choose any vertices $u_0,\ldots,u_{s_i-1} \in X - \bigcup_{j=1}^{i-1} V(L_j)$ and $v_0,\ldots,v_{s_i-1} \in Y - \bigcup_{j=1}^{i-1} V(L_j) - \bigcup_{j=i+1}^k V(D_j^{(0)})$. Then  let $L_i$ be the directed cycle
    $$L_i=u_0 \, v_0 \, u_1 \, v_1 \, \ldots \, u_{s_i-1} \, v_{s_i-1} \, u_0.$$
\item Otherwise, we have $0 < t_i < m_i$, and $D_i^{(0)}$ is a directed $t_i$-path, say $D_i^{(0)}=v_0 \, v_1 \, \ldots \, v_{t_i}$, for some
    $v_0, v_1, \ldots, v_{t_i} \in Y$. Choose any vertices $u_0,\ldots,u_{s_i-1} \in X - \bigcup_{j=1}^{i-1} V(L_j)$ and $v_{t_i+1}, v_{t_i+2}, \ldots, v_{t_i+s_i-1} \in Y - \bigcup_{j=1}^{i-1} V(L_j) - \bigcup_{j=i}^k V(D_j^{(0)})$. Then  let $L_i$ be the directed cycle
    $$L_i= D_i^{(0)} \, v_{t_i} \, u_0 \, v_{t_i+1} \, u_1 \, v_{t_i+2} \, u_2 \, \ldots \, v_{t_i+s_i-1} \, u_{s_i-1} \, v_0.$$
\end{enumerate}
Note that since $\sum_{i=1}^k s_i=s$ and
$$\sum_{i=1}^k (s_i+t_i)= \sum_{i=1}^k (m_i-s_i)= \sum_{i=1}^k m_i - \sum_{i=1}^k s_i=(s+t)-s=t,$$ the required vertices can indeed be found.

Observe that in all three cases, for each $i=1,2,\ldots,k$, the constructed digraph $L_i$ is a directed $m_i$-cycle with $s_i$ vertices in $X$ and $s_i+t_i$ vertices in $Y$. The digraphs $L_1,\ldots,L_k$ are pairwise disjoint, so $F_0=L_1 \cup \ldots \cup L_k$ is a $(\vec{C}_{m_1},\ldots,\vec{C}_{m_k})$-factor of $D$. For $j \in \ZZ_t$, obtain $F_j$ from $F_0$ by first applying $\rho^j$ to $F_0$, and then for each $i=1,2,\ldots,k$ such that $t_i>0$, replacing $\rho^j(D_i^{(0)})$ with
$D_i^{(j)}$. By Assumption (d), without loss of generality,  we have that $V(D_i^{(j)})=V(\rho^j(D_i^{(0)}))$ for each $i$, and if $D_i^{(0)}$ is a directed path, then $D_i^{(j)}$ is directed path with the same source and the same terminus as $\rho^j(D_i^{(0)})$. It follows that $F_j$ is also a $(\vec{C}_{m_1},\ldots,\vec{C}_{m_k})$-factor of $D$.

We claim that $\F =\{ F_j: j \in \ZZ_t \}$ is a $(\vec{C}_{m_1},\ldots,\vec{C}_{m_k})$-factorization of $\bar{K}_s \join D''$. Observe that for each $j \in \ZZ_t$,
$$A(F_j)= \rho^j( A_0) \cup A(H_j),$$
where $A_0= \left( A(F_0) \cap (X  \times Y) \right) \cup \left( A(F_0) \cap (Y  \times X) \right).$
For any vertex $x \in X$, the indegree and outdegree of $x$ in $F_0$ is 1. Therefore each arc incident with $x$ is covered exactly once in $\bigcup_{j=0}^{t-1} \rho^j(A_0)$. Since by assumption $\D=\{ H_0,\ldots,H_{t-1} \}$ decomposes $D''$, it follows that $\F =\{ F_j: j \in \ZZ_t \}$ decomposes $\bar{K}_s \join D''$.

Finally, let $\{ F_1^{[s]},\ldots, F_{s-1}^{[s]} \}$ and $\{ F_1^{[t]},\ldots, F_{s-1}^{[t]} \}$
be a $(\vec{C}_{m_1},\ldots,\vec{C}_{m_{\ell}})$-factorization of $\bar{K}_s^\ast$ and a $(\vec{C}_{m_{\ell+1}},\ldots,\vec{C}_{m_k})$-factorization of $D'$, respectively.
 Then
$$\F'=\{ F_i^{[s]} \cup F_i^{[t]}: i=1,2,\ldots,s-1 \}$$
is a $(\vec{C}_{m_{1}},\ldots,\vec{C}_{m_k})$-factorization of $K_s^\ast \du D'$, and $\F \cup \F'$ is a $(\vec{C}_{m_{1}},\ldots,\vec{C}_{m_k})$-factorization of $D$.
\end{proof}


Corollary~\ref{cor:main} below is a simpler (but slightly more limited) version of Proposition~\ref{prop:main}. For most of our recursive constructions, Corollary~\ref{cor:main} will suffice; however, to solve OP$^\ast(2,2,4)$ (see Lemma~\ref{lem:(2,2,4)} below), Proposition~\ref{prop:main} will be required in full generality.

\begin{cor}\label{cor:main}
Let $\ell$, $k$, and $m_1,\ldots, m_k$ be integers such that $1 \le \ell < k$ and $m_i \ge 2$ for $i=1,\ldots, k$. Let $s=m_1+\ldots+m_{\ell}$ and $t=m_{\ell+1}+\ldots+m_{k}$, and assume $s<t$.

Then OP$^\ast(m_1,\ldots,m_k)$ has a solution if the following conditions all hold.
\begin{enumerate}[(1)]
\item OP$^\ast(m_1,\ldots,m_{\ell})$ has a solution.
\item There exists a set $S \subseteq \ZZ_t^\ast$ and non-negative integers $s_1, \ldots, s_k, t_1, \ldots, t_k$ such that:
   \begin{enumerate}[(a)]
    \item $\dci(t;\ZZ_t^\ast -S)$ admits a $(\vec{C}_{m_{\ell+1}},\ldots,\vec{C}_{m_{k}})$-factorization;
    \item $s_1+ \ldots + s_k=s$;
    \item $m_i=2s_i+t_i$  for $i=1,2,\ldots,k$; and
    \item $\dci(t;S)$ admits an $S$-orthogonal $(D_1,\ldots,D_k)$-subdigraph such that, for all $i=1,2,\ldots,k$, we have that $D_i \cong \vec{P}_{t_i}$ if $t_i<m_i$, and $D_i \cong \vec{C}_{m_i}$ if $t_i=m_i$.
    \end{enumerate}
\end{enumerate}
\end{cor}

\begin{proof}
Assume Conditions (1) and (2) hold, label the vertices of $D=K_s^\ast \join K_t^\ast$ as in the proof of Proposition~\ref{prop:main}, and let $\rho=(y_0 \, y_1 \, \ldots \, y_{t-1} )$. Let $D'=\dci(t;\ZZ_t^\ast -S)$ and $D''=\dci(t;S)$, so $K_t^\ast=D' \oplus D''$.
Let $H$ be an $S$-orthogonal $(D_1,\ldots,D_k)$-subdigraph of $\dci(t;S)$ satisfying Condition (2d). Then $\D=\{ \rho^j(H): j \in \ZZ_t \}$ is a decomposition of $\dci(t;S)$ satisfying Condition (2d) of Proposition~\ref{prop:main}.

Observe that
$$|A(H)|=\sum_{i=1}^k t_i=\sum_{i=1}^k (m_i-2s_i)=\sum_{i=1}^k m_i -2\sum_{i=1}^k s_i= (s+t)-2s=t-s, $$
and so $|S|=t-s$ and $|\ZZ_t^\ast-S|=t-1-(t-s)=s-1$. Thus any  $(\vec{C}_{m_{\ell+1}},\ldots,\vec{C}_{m_{k}})$-factorization of $D'$ indeed contains exactly $s-1$ directed 2-factors.

By Proposition~\ref{prop:main}, it follows that OP$^\ast(m_1,\ldots,m_k)$ has a solution.
\end{proof}

\begin{lemma}\label{lem:(2,2,4)}
OP$^\ast(2,2,4)$ has a solution.
\end{lemma}

\begin{proof}
We use Proposition~\ref{prop:main} with $m_1=m_2=2$, $m_3=4$, $\ell=1$ and $k=3$. Hence $s=2$ and $t=6$. By Proposition~\ref{prop:main}, since OP$^\ast(2)$ has a solution, it suffices to find a decomposition $\{ D',D''\}$ of $K_6^\ast$ and non-negative integers $s_1,s_2,s_3,t_1,t_2,t_3$ satisfying Condition~(2). We take $s_1=s_2=1$, $s_3=0$, $t_1=t_2=0$, and $t_3=4$, so Conditions (2b) and (2c) hold.

Let  $V(K_6^\ast)=\{ y_0,\ldots,y_5 \}$ and  $\rho=( y_0 \, y_1 \, \ldots \, y_5)$. Let
$$D'=y_1 \, y_4 \, y_1 \;  \cup \; y_0 \, y_5 \, y_2 \, y_3 \, y_0,$$
and let $D''$ be its complement in $K_6^\ast$. Note that $D'$ is a $(\vec{C}_2,\vec{C}_4)$-factor of $K_6^\ast$, thus satisfying Condition~(2a).

Define the following $(\vec{P}_0,\vec{P}_0,\vec{C}_4)$-subdigraphs of $D''$:
\begin{eqnarray*}
H_0 &= y_0 \;\cup\; y_1 \;\cup\; y_2 \, y_5 \, y_3 \, y_4 \, y_2, \qquad\qquad\qquad H_3 &=\rho(H_2), \\
H_1 &= y_1 \;\cup\; y_2 \;\cup\; y_3 \, y_5 \, y_4 \, y_0 \, y_3, \qquad\qquad\qquad
H_4 &=\rho^2(H_2) \\
H_2 &= y_2 \;\cup\; y_3 \;\cup\; y_4 \, y_5 \, y_1 \, y_0 \, y_4, \qquad\qquad\qquad
H_5 &=\rho^3(H_2).
\end{eqnarray*}
It is not difficult to verify that $\{ H_i: i \in \ZZ_6 \}$ is a decomposition of $D''$ satisfying Condition~(2d) of Proposition~\ref{prop:main}. Hence OP$^\ast(2,2,4)$ has a solution.
\end{proof}


\section{Extending the 2-factor by a single cycle: \\ technical lemmas}

In this section, we shall accomplish the heavy technical work that, using Corollary~\ref{cor:main}, will allow us to obtain a solution to OP$^\ast(m_1,\ldots,m_{\ell},t)$ from
a solution to OP$^\ast(m_1,\ldots,m_{\ell})$  when $t$ is sufficiently large; see Theorem~\ref{the:main-ext}(1).

\begin{lemma}\label{lem:long1}
Let $s$ and $t$  be integers such that $2 \le s < t$, and $s \ne 4$ if $t$ is even. Furthermore,
let $$\textstyle D=\{ \pm 1, \pm 2, \ldots, \pm \frac{s-1}{2} \}$$
if $s$ is odd, and
$$\textstyle D=\{ -1 \} \cup \{ \pm 2, \pm 3, \ldots, \pm \frac{s}{2} \}$$
if $s$ is even.
Then $\dci(t;D)$ admits a $\vec{C}_t$-factorization.
\end{lemma}

\begin{proof}
Let $k=\lfloor \frac{s}{2} \rfloor$ and $T=\{ 2,3, \ldots, k \}$.
If $k \ge 3$, then $T$ has a partition
$$\textstyle \P = \left\{ \{ 2,3 \}, \{ 4,5 \}, \ldots, \{ k-1, k \} \right\}$$
or
$$\textstyle \P = \left\{ \{ 2,3,4 \}, \{ 5,6 \}, \ldots, \{ k-1, k \} \right\}.$$
In either case, for each $S \in \P$, the circulant graph $\ci(t;S \cup (-S))$ has a decomposition into Hamilton cycles. This conclusion follows from Theorem~\ref{the:BerFavMah} if $|S|=2$, and from Theorem~\ref{the:Wes} if $|S|=3$ and $t$ is even. If $|S|=3$ and $t$ is odd, then we first write $\{ 2,3,4 \}=\{ 2 \} \cup \{ 3,4 \}$, and use a similar reasoning. Directing each $t$-cycle in this decomposition of $\ci(t;S \cup (-S))$ in both directions, we obtain a $\vec{C}_t$-factorization of  $\dci(t;S \cup (-S))$. Since $\dci(t;D)$ decomposes into directed circulants of this form and either $\dci(t; \{ -1 \})$ or $\dci(t; \{ \pm 1 \})$, each of which also admits a $\vec{C}_t$-factorization, the result follows.

If $k \le 2$, then $2 \le s \le 5$. Since $t$ is odd if $s=4$, it can be established similarly that in each of these cases $\dci(t;D)$ admits a $\vec{C}_t$-factorization.
\end{proof}

\begin{lemma}\label{lem:long2}
Let $a$, $s$, and $t$ be integers such that $2 \le s < t$, $s \ne 3$ if $t$ is even, $0 \le a \le \min \{ \lfloor \frac{s}{3} \rfloor, 2\lfloor \frac{t}{2} \rfloor -s \}$, and $a \equiv s \pmod{2}$.
Furthermore, let $$\textstyle S=\{ \pm \frac{s+1}{2}, \pm \frac{s+3}{2}, \ldots, \pm \lfloor \frac{t}{2} \rfloor \}$$
if $s$ is odd, and
$$\textstyle S=\{ 1 \}  \cup \{ \pm (\frac{s}{2}+1), \pm (\frac{s}{2}+2), \ldots, \pm \lfloor \frac{t}{2} \rfloor \}$$
if $s$ is even.

Then $\dci(t;S)$ admits an $S$-orthogonal $(\vec{P}_1^{\langle a \rangle},\vec{P}_{t-s-a})$-subdigraph.
\end{lemma}

\begin{proof}
First observe that in all cases, $|S|=t-s$. Let the vertex set of the circulant digraph $\dci(t;S)$ be $V=\{ y_i: i \in \ZZ_t \}$.

For any $0 \le \ell \le \frac{t-s-1}{2}$, let $R$ be any subset of $S$ of the form
$$\textstyle R=\{ \pm d_i: i=1,2,\ldots,\ell \} \cup \{ \lfloor \frac{t}{2} \rfloor \} \quad \mbox{with }
1 \le d_1 < d_2 < \ldots < d_\ell < \lfloor \frac{t}{2} \rfloor. \eqno(\star)$$
For any $k \in \ZZ_t$, we then define a directed walk $P(R,y_k)$ as follows:
\begin{eqnarray*}
P(R,y_k) &=& y_k \, y_{k+d_1} \, y_{k+d_1-d_2} \, y_{k+d_1-d_2+d_3} \, y_{k+d_1-d_2+d_3-d_4} \, \ldots \, y_{k+\sum_{i=1}^\ell (-1)^{i-1} d_i} \, \\
&& y_{k+\sum_{i=1}^\ell (-1)^{i-1} d_i+ \lfloor \frac{t}{2} \rfloor}
\, y_{k+\sum_{i=1}^{\ell-1} (-1)^{i-1} d_i+\lfloor \frac{t}{2} \rfloor}
\, y_{k+\sum_{i=1}^{\ell-2} (-1)^{i-1} d_i+\lfloor \frac{t}{2} \rfloor} \ldots
\, y_{k+d_1+\lfloor \frac{t}{2} \rfloor}
\, y_{k+\lfloor \frac{t}{2} \rfloor}.
\end{eqnarray*}
Observe that $P(R,y_k)$ is actually a directed path that successively traverses arcs of the following differences:
$$\textstyle d_1, -d_2, d_3, -d_4, \ldots, (-1)^{\ell-1}d_{\ell}, \lfloor \frac{t}{2} \rfloor, (-1)^{\ell}d_{\ell}, (-1)^{\ell-1}d_{\ell-1}, \ldots, d_2, -d_1.$$
Moreover, its vertex set is contained in the set
$$\textstyle \{ y_i: k+d_1 \le i \le k+\lfloor \frac{t}{2} \rfloor \} \cup \{ y_i: k+d_1-\lceil \frac{t}{2} \rceil \le i \le k \}.$$

\medskip

{\sc Case 1}: $t$ and $s$ are both even. Then $S=\{ 1 , \pm (\frac{s}{2}+1), \pm (\frac{s}{2}+2), \ldots, \pm (\frac{t}{2}-1), \frac{t}{2}\}.$ For $i=-\frac{a}{2}+1, -\frac{a}{2}+2,\ldots,0,1,\ldots,\frac{a}{2}$,
define the  directed 1-path
$$\textstyle A_i=y_i\, y_{\frac{t}{2}-i+1}.$$
Observe that the arc in $A_i$ is of difference $\frac{t}{2}-(2i-1)$ if $i > 0$, and of difference $-(\frac{t}{2}+(2i-1))$ if $i \le 0$. Thus $A_{-\frac{a}{2}+1},\ldots,A_{\frac{a}{2}}$ jointly contain exactly one arc of each difference in
$$\textstyle T=\{ \pm (\frac{t}{2}-1), \pm (\frac{t}{2}-3),\ldots, \pm (\frac{t}{2}-(a-1)) \}.$$
 Moreover, the $A_i$ are pairwise disjoint, so $A= \bigcup_{i=-\frac{a}{2}+1}^{\frac{a}{2}}$ is a $T$-orthogonal $(\vec{P}_1^{\langle a \rangle})$-subdigraph of $\dci(t; \ZZ_t^\ast)$. Its vertex set is
 $$\textstyle V(A)=\{ y_i: -\frac{a}{2}+1 \le i \le \frac{a}{2} \} \cup
\{ y_i: \frac{t}{2}-\frac{a}{2}+1 \le i \le \frac{t}{2}+\frac{a}{2} \}.$$

\smallskip

\begin{figure}[t]
\centerline{\includegraphics[scale=0.6]{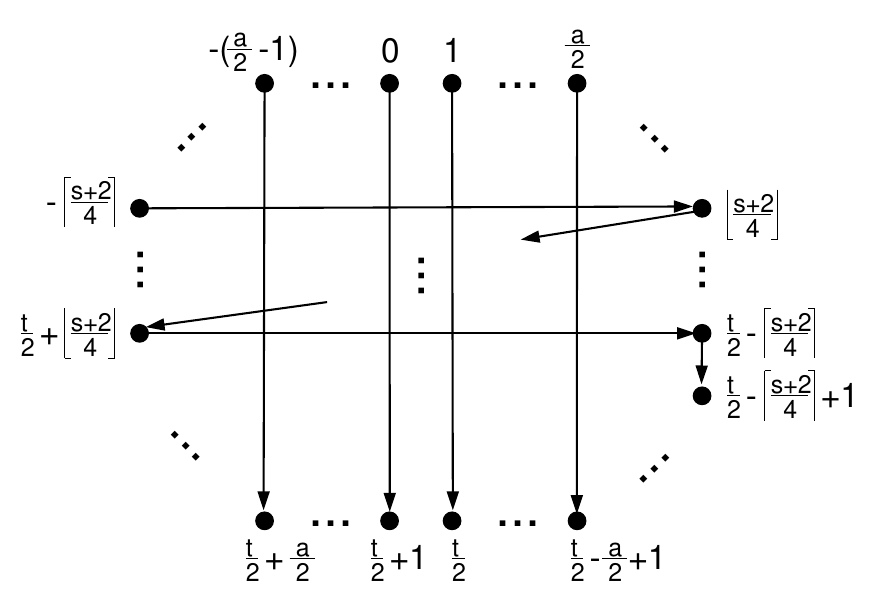}}
\caption{Lemma~\ref{lem:long2}, Subcase 1.1. (Only the subscripts of the vertices are specified.)}
\label{fig:Lee1}
\end{figure}

{\sc Subcase 1.1}: $a \le \frac{t-s}{2}-1$. If $a>0$, then
$T \subseteq S$ and $\frac{s+2}{2} \not\in T$.
Let $R=S-T-\{ 1 \}$. Then the directed path $P(R,y_{-\lceil \frac{s+2}{4} \rceil})$ is well defined, with $d_1=\frac{s+2}{2}$. Its vertex set is contained in
$$\textstyle \{ y_i: \lfloor \frac{s+2}{4} \rfloor \le i \le \frac{t}{2}-\lceil \frac{s+2}{4} \rceil\} \cup
\{ y_i: \frac{t}{2}+\lfloor \frac{s+2}{4} \rfloor \le i \le t-\lceil \frac{s+2}{4} \rceil\},$$
and its terminus is $y_{\frac{t}{2}-\lceil \frac{s+2}{4} \rceil}$. Let $P=P(R,y_{-\lceil \frac{s+2}{4} \rceil})\, y_{\frac{t}{2}-\lceil \frac{s+2}{4} \rceil}\, y_{\frac{t}{2}-\lceil \frac{s+2}{4} \rceil+1}$.

Since $a \le \frac{s}{3}$, we have $\frac{a}{2} < \lfloor \frac{s+2}{4} \rfloor$. It follows that $P$ and $A$ are disjoint, and that $D=A \cup P$
is an $S$-orthogonal $(\vec{P}_1^{\langle a \rangle},\vec{P}_{t-s-a})$-subdigraph of $\dci(t;S)$. See Figure~\ref{fig:Lee1}.

If $a=0$, then $T=\emptyset$ and $D=P$ is the required $S$-orthogonal $(\vec{P}_{t-s})$-subdigraph of $\dci(t;S)$.

\smallskip

\begin{figure}[t]
\centerline{\includegraphics[scale=0.6]{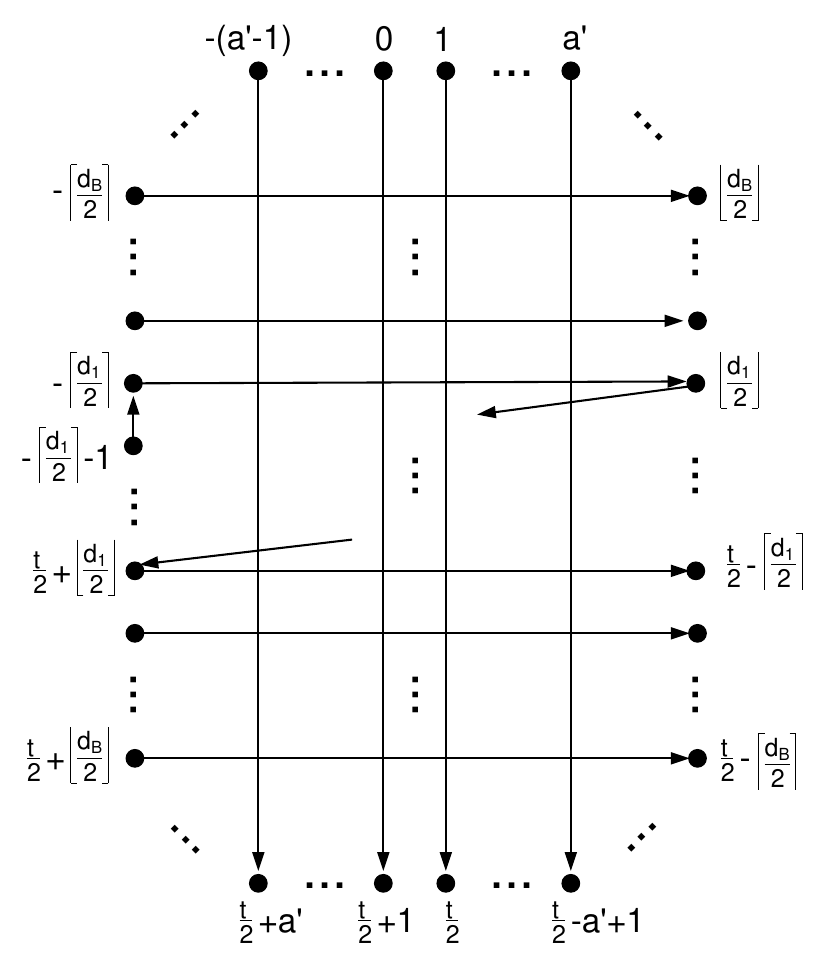}}
\caption{Lemma~\ref{lem:long2}, Subcase 1.2, for $|R|>1$. (Only the subscripts of the vertices are specified.)}
\label{fig:Lee2}
\end{figure}

{\sc Subcase 1.2}: $\frac{t-s}{2} \le a < t-s$. Let $a'=\lfloor \frac{t-s}{4} \rfloor$ and $d_A=\frac{t}{2}-(2a'-1)$. Observe that $d_A=\frac{s+2}{2}$ if $t \equiv s \pmod{4}$, and $d_A=\frac{s+4}{2}$ otherwise. Let
$$\textstyle T'=\{ \pm (\frac{t}{2}-1), \pm (\frac{t}{2}-3),\ldots, \pm  d_A \},$$
 so that $|T'|=2a'$. Let $A'=\bigcup_{i=-a'+1}^{a'} A_i$. Then $A'$ is a $T'$-orthogonal $(\vec{P}_1^{\langle 2a' \rangle})$-subdigraph of $\dci(t;T')$. Its vertex set is
$$\textstyle V(A')=\{ y_i: -a'+1 \le i \le a' \} \cup
\{ y_i: \frac{t}{2}-a'+1 \le i \le \frac{t}{2}+a' \}.$$

Let $b=a-2a'$, and note that $b$ is even. Let $d_B=\frac{s+4}{2}$ if $t \equiv s \pmod{4}$, and $d_B=\frac{s+2}{2}$ otherwise, so that $\{ d_A,d_B \} = \{ \frac{s+2}{2}, \frac{s+4}{2} \}$ in both cases.
Let
$$R'=\{ \pm d_B, \pm (d_B+2), \ldots, \pm (d_B+b-2) \},$$
 so that $R' \cap T' =\emptyset$. Since $a<t-s$, we know that $d_B+b-2 < \frac{t}{2}$ and $|R'|=b$.

For each $d \in R'$ such that $d>0$, define the  directed 1-paths
$$\textstyle B_d=y_{-\lceil \frac{d}{2} \rceil} \, y_{\lfloor \frac{d}{2} \rfloor} \qquad \mbox{ and }\qquad
B_{-d}=y_{\frac{t}{2}+\lfloor \frac{d}{2} \rfloor} \,  y_{\frac{t}{2}-\lceil \frac{d}{2} \rceil}.$$
Observe that the arcs in $B_d$ and $B_{-d}$ are of difference $d$ and $-d$, respectively. Thus, the directed 1-paths in $\{ B_i: i \in R' \}$  jointly contain exactly one arc of each difference in $R'$. Moreover, the digraphs $B_i$ are pairwise disjoint, so that
$B= \bigcup_{i \in R'} B_i$ is an $R'$-orthogonal $(\vec{P}_1^{\langle b \rangle})$-subdigraph of $\dci(t; R')$.

Let $R=S-T'-R'-\{ 1 \}$, and observe that $\frac{t}{2} \in R$.  Assume first that $|R|>1$. Then the directed path $P(R,y_{-\lceil \frac{d_1}{2} \rceil})$ is well defined, and
with the notation ($\star)$, we have $d_1=d_B+b$.  If $\ell \ge 2$, then $d_2 =d_1+2$, and if $\ell=1$, then $\frac{t}{2}=d_1+2$. In either case, vertex $y_{-\lceil \frac{d_1}{2} \rceil-1}$ is not on $P(R,y_{-\lceil \frac{d_1}{2} \rceil})$, and $P=y_{-\lceil \frac{d_1}{2} \rceil-1} \, y_{-\lceil \frac{d_1}{2} \rceil} \, P(R,y_{-\lceil \frac{d_1}{2} \rceil})$ is a directed path. Moreover, $P$ is disjoint from $B$, and $V(B \cup P)$ is contained in
$$\textstyle \{ y_i: \lfloor \frac{d_B}{2} \rfloor \le i \le \frac{t}{2} - \lceil \frac{d_B}{2} \rceil \} \cup
\{ y_i: \frac{t}{2}+ \lfloor \frac{d_B}{2} \rfloor \le i \le t-\lceil \frac{d_B}{2} \rceil\}.$$

Since $\frac{t-s}{2} \le a \le \frac{s}{3}$, it is not difficult to show that
$a' < \lfloor \frac{d_B}{2} \rfloor$.
It follows that $B \cup P$ and $A'$ are disjoint, and that $D=A' \cup B \cup P  $
is an $S$-orthogonal $(\vec{P}_1^{\langle a \rangle},\vec{P}_{t-s-a})$-subdigraph of $\dci(t;S)$. See Figure~\ref{fig:Lee2}.

If, however,  $|R|=1$, then $P=y_{-\lceil \frac{t}{4} \rceil} \, y_{\lfloor \frac{t}{4} \rfloor}$ is a directed 1-path, and $t-s-a=2$. Let $A''= \bigcup_{i=-a'+2}^{a'} A_i$ and $P'=y_{-a'} \, y_{-a'+1} A_{-a'+1}$. As $a' < \lceil \frac{d_B}{2} \rceil$, we have that $P'$ is a directed 2-path disjoint from $A'' \cup B \cup P $. It follows that
$D=A'' \cup B \cup P \cup P'$ is an $S$-orthogonal $(\vec{P}_1^{\langle a \rangle},\vec{P}_{t-s-a})$-subdigraph of $\dci(t;S)$.

\smallskip

{\sc Subcase 1.3}: $a = t-s$. This case is very similar to Subcase 1.2, except that $|R'|=b-1$ and $R=\emptyset$. We define the unions $A'$ and $B$ of directed 1-paths exactly as in Subcase 1.2, while the last directed 1-path will be $P=y_{-a'-1} \, y_{-a'}$. As $t-s=a\le\frac{s}{3}<s$, it can be shown that $a'+1 < \lceil \frac{d_B}{2} \rceil$. It follows that $P$ is disjoint from $A' \cup B$, and  $D=A' \cup B \cup P $ is an $S$-orthogonal $(\vec{P}_1^{\langle a \rangle})$-subdigraph of $\dci(t;S)$. Hence $\dci(t;S)$ admits an $S$-orthogonal $(\vec{P}_1^{\langle a \rangle},\vec{P}_0)$-subdigraph as well.

\bigskip

{\sc Case 2}: $t$ is even and $s$ is odd. Then $S=\{ \pm \frac{s+1}{2}, \pm \frac{s+3}{2}, \ldots, \pm (\frac{t}{2}-1), \frac{t}{2}\}.$ By the assumption, we have $s \ne 3$.

\begin{figure}[t]
\centerline{\includegraphics[scale=0.6]{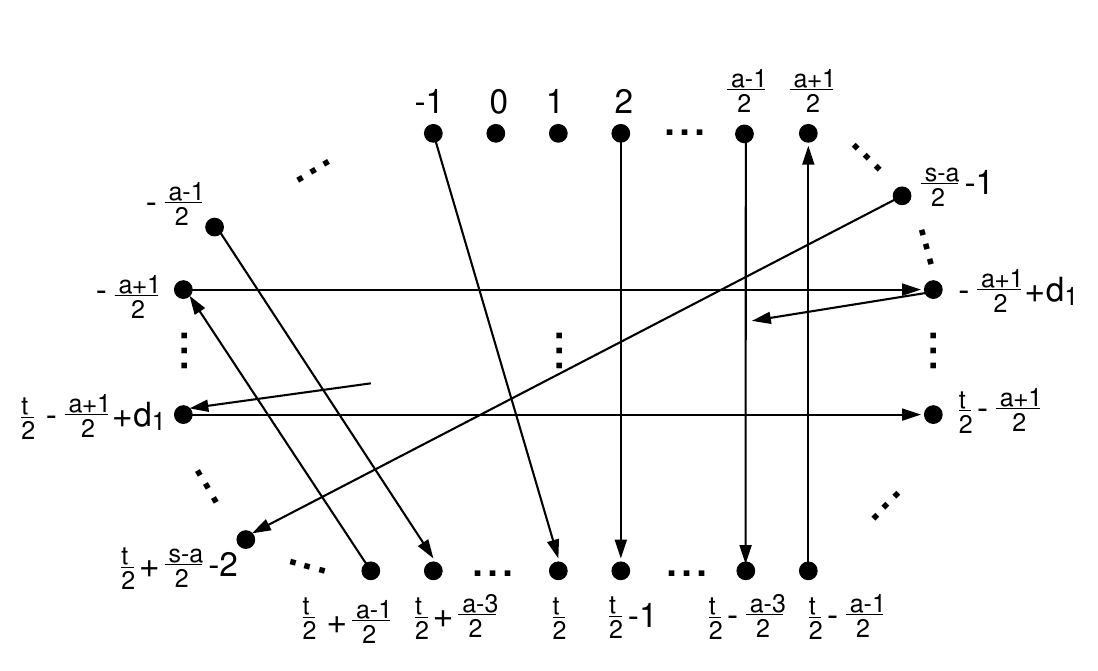}}
\caption{Lemma~\ref{lem:long2}, Subcase 2.1. (Only the subscripts of the vertices are specified.)}
\label{fig:Leo1}
\end{figure}

{\sc Subcase 2.1}: $a \le \frac{t-s-1}{2}$ and $(s,a) \not\in \{ (9,3),(5,1) \}$.
Define directed 1-paths
$$\textstyle A_{i}=y_{i}\, y_{\frac{t}{2}-i-1} \quad \mbox{ for }
i=-1,-2,\ldots,-\frac{a-1}{2},$$
and
$$\textstyle A_{i}=y_{i}\, y_{\frac{t}{2}-i+1} \quad \mbox{ for }
i=2,3,\ldots,\frac{a-1}{2}.$$
In addition, let
$$ A_{1}= y_{\frac{s-a}{2}-1} \, y_{\frac{t}{2}+\frac{s-a}{2}-2}
\qquad \mbox{ and } \qquad A_{-\frac{a+1}{2}}= y_{\frac{t}{2}-\frac{a-1}{2}}\, y_{\frac{a+1}{2}}.$$
Observe that the arc in $A_i$ is of difference $\frac{t}{2}-(2i-1)$ if $i>0$, and of difference $-(\frac{t}{2}+(2i+1))$ if $i < 0$.
Thus the $A_{i}$, for $i \in \{ \pm 1, \ldots,\pm \frac{a-1}{2}, -\frac{a+1}{2} \}$,
jointly contain exactly one arc of each difference in
$$\textstyle T=\{ \pm (\frac{t}{2}-1), \pm (\frac{t}{2}-3),\ldots, \pm (\frac{t}{2}-(a-2)), -(\frac{t}{2}-a) \}.$$
 Moreover, since $a \le \frac{s}{3}$ and $(s,a) \not\in \{ (9,3),(5,1),(3,1) \}$, we find that
$\frac{a+1}{2} < \frac{s-a}{2}-1$. Consequently,
the $A_i$ are pairwise disjoint, and $A= \bigcup_{i=1}^{\frac{a-1}{2}} A_{-i} \cup \bigcup_{i=1}^{\frac{a+1}{2}} A_{i}$  is a $T$-orthogonal $(\vec{P}_1^{\langle a \rangle})$-subdigraph of $\dci(t; T)$. Its vertex set is contained in
$$\textstyle \{ y_i: -\frac{a-1}{2} \le i \le \frac{s-a}{2}-1 \} \cup
\{ y_i: \frac{t}{2}-\frac{a-1}{2} \le i \le \frac{t}{2}+\frac{a-3}{2} \} \cup \{ y_{\frac{t}{2}+\frac{s-a}{2}-2} \}.$$

Since $a \le \frac{t-s-1}{2}$, we know that
$T \subseteq S$.
Let $R=S-T-\{ \frac{t}{2}-a \}$. Then the directed path $P(R,y_{- \frac{a+1}{2}})$ is well defined, with $d_1 \ge \frac{s+1}{2}$. Its vertex set is contained in
$$\textstyle \{ y_i: d_1-\frac{a+1}{2}  \le i \le \frac{t}{2}- \frac{a+1}{2} \} \cup
\{ y_i: \frac{t}{2}+ d_1-\frac{a+1}{2}  \le i \le t- \frac{a+1}{2} \}.$$
Let $P=y_{\frac{t}{2}+ \frac{a-1}{2}} \, y_{- \frac{a+1}{2}} \, P(R,y_{- \frac{a+1}{2}})$.
Observe that $\frac{s-a}{2}= \frac{s+1}{2}-\frac{a+1}{2} \le d_1 -\frac{a+1}{2}$.
It follows that $P$ is an $(R \cup \{ \frac{t}{2}-a \})$-orthogonal directed path, and that $P$ and $A$ are disjoint. Hence $D=A \cup P$
is an $S$-orthogonal $(\vec{P}_1^{\langle a \rangle},\vec{P}_{t-s-a})$-subdigraph of $\dci(t;S)$. See Figure~\ref{fig:Leo1}.

\smallskip

{\sc Subcase 2.2}: $(s,a) \in \{ (9,3),(5,1) \}$.

If $s=9$ and $a=3$, then $S=\{ \pm 5, \pm 6, \ldots, \pm (\frac{t}{2}-1), \frac{t}{2}\}$. Construct directed paths $A_{-1}$, $A_{-2}$, and $P$ as in Subcase 2.1, but let
$ A_1=y_{\frac{t}{2}+2} \, y_1$.
The rest of the construction is completed as in Subcase 2.1.

If $s=5$ and $a=1$, then $S=\{ \pm 3, \pm 4, \ldots, \pm (\frac{t}{2}-1), \frac{t}{2}\}$. Define directed paths $A=y_0 \, y_{\frac{t}{2}+1}$ and $P=y_{\frac{t}{2}} \, y_{-1} \, P(R,y_{-1})$, where $R=S- \{ \pm(\frac{t}{2}-1) \}$, and complete the proof as in Subcase 2.1.

\smallskip

\begin{figure}[t]
\centerline{\includegraphics[scale=0.6]{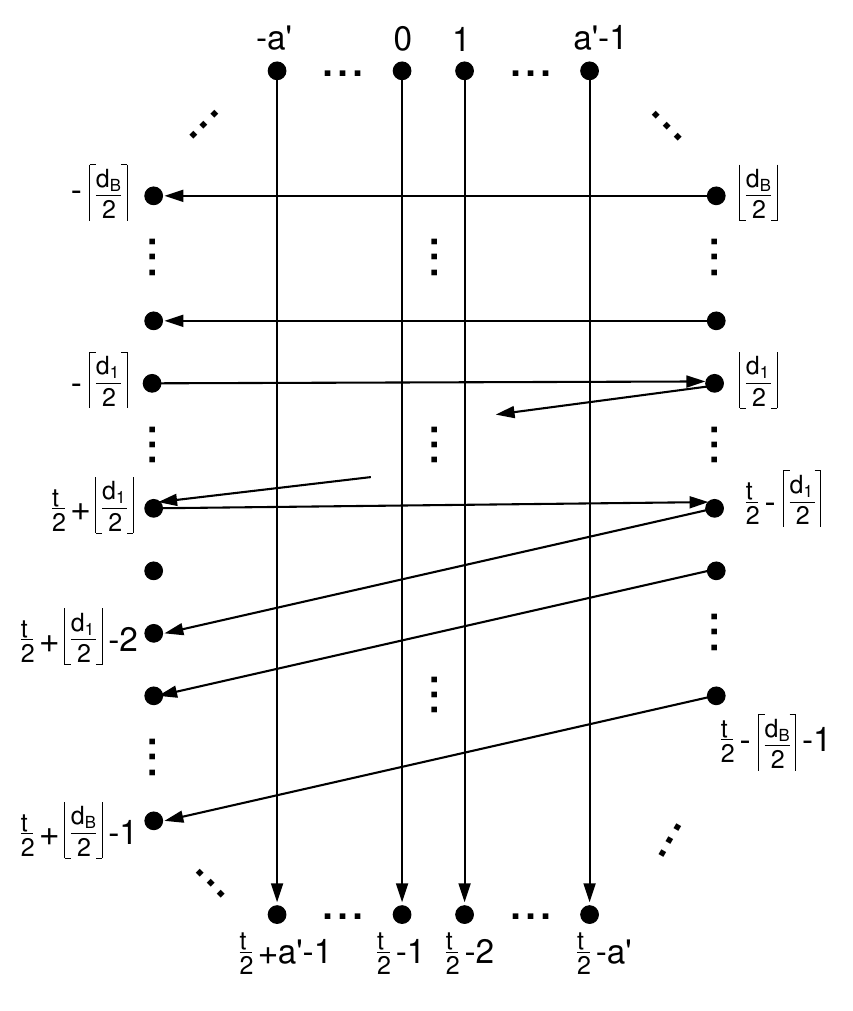}}
\caption{Lemma~\ref{lem:long2}, Subcase 2.3, with $|R|>1$. (Only the subscripts of the vertices are specified.)}
\label{fig:Leo2}
\end{figure}

{\sc Subcase 2.3}: $\frac{t-s+1}{2} \le a < t-s$. Let $a'=\lfloor \frac{t-s+1}{4} \rfloor$ and $d_A=\frac{t}{2}-(2a'-1)$. Observe that $d_A=\frac{s+1}{2}$ if $t \equiv s+3 \pmod{4}$, and $d_A=\frac{s+3}{2}$ otherwise. Let
$$\textstyle T'=\{ \pm (\frac{t}{2}-1), \pm (\frac{t}{2}-3),\ldots, \pm  d_A \},$$
 so that $|T'|=2a'$.

Define directed 1-paths
$$\textstyle A_{i}=y_{i}\, y_{\frac{t}{2}-i-1} \quad \mbox{ for }
i=-a',-a'+1,\ldots, 0,1,\ldots,a'-1.$$
Observe that the arc in $A_i$ is of difference $\frac{t}{2}-(2i+1)$ if $i \ge 0$, and of difference $-(\frac{t}{2}+(2i+1))$ if $i < 0$. Thus the $A_{i}$, for $i \in \{ 0, \pm 1, \ldots,\pm (a'-1),-a' \}$,
jointly contain exactly one arc of each difference in $T'$.
The $A_i$ are clearly pairwise disjoint, and $A= \bigcup_{i=-a'}^{a'-1} A_{i}$  is a $T'$-orthogonal $(\vec{P}_1^{\langle 2a' \rangle})$-subdigraph of $\dci(t; T')$. Its vertex set is
$$\textstyle V(A)=\{ y_i: -a' \le i \le a'-1 \} \cup
\{ y_i: \frac{t}{2}-a' \le i \le \frac{t}{2}+a'-1 \}.$$

Let $b=a-2a'$, and note that $b$ is odd, and $b \ge 0$ as $a \ge \frac{t-s+1}{2}$.
Let $d_B=\frac{s+1}{2}$ if $t \equiv s+1 \pmod{4}$, and $d_B=\frac{s+3}{2}$ otherwise, so that $\{ d_A,d_B \} = \{ \frac{s+1}{2}, \frac{s+3}{2} \}$ in both cases.
Let
$$R'=\{ \pm d_B, \pm (d_B+2), \ldots, \pm (d_B+b-1) \},$$
and note that $R' \cap T' =\emptyset$. As $a< t-s$, we can see that $d_B+b-1<\frac{t}{2}$. Hence $|R'|=b+1$.

For $d \in R'$ such that $d>0$, define  directed 1-paths
$$\textstyle B_{-d}=y_{\lfloor \frac{d}{2} \rfloor} \, y_{-\lceil \frac{d}{2} \rceil}   \qquad \mbox{ and }\qquad
B_{d}= y_{\frac{t}{2}-\lceil \frac{d+2}{2} \rceil} \, y_{\frac{t}{2}+\lfloor \frac{d-2}{2} \rfloor}.$$
Observe that the arcs in $B_{-d}$ and $B_{d}$ are of difference $-d$ and $d$, respectively. Thus, the directed 1-paths in $\{ B_i: i \in R' \}$  jointly contain exactly one arc of each difference in $R'$. Moreover, the digraphs $B_i$ are pairwise disjoint. Let $R''=R'-\{ d_B+b-1\}$. Then
$B= \bigcup_{i \in R''} B_i$ is an $R''$-orthogonal $(\vec{P}_1^{\langle b \rangle})$-subdigraph of $\dci(t; R'')$.

Let $R=S-T'-R'$, and observe that $\frac{t}{2} \in R$. If $|R|>1$, then with the notation $(\star)$, we have $d_1=d_B+b+1$, and
the directed path $P(R,y_{-\lceil \frac{d_1}{2} \rceil})$ is well defined.  Its last two vertices are $y_{\frac{t}{2}+\lfloor \frac{d_1}{2} \rfloor}$ and  $y_{\frac{t}{2}-\lceil \frac{d_1}{2} \rceil}$. Hence
$P(R,y_{-\lceil \frac{d_1}{2} \rceil})$ is disjoint from $B$, but shares the vertex $y_{\frac{t}{2}-\lceil \frac{d_1}{2} \rceil}=y_{\frac{t}{2}-\lceil \frac{d_B+b+1}{2} \rceil}$ with $B_{d_B+b-1}$. Hence
 $P=P(R,y_{-\lceil \frac{d_1}{2} \rceil})\, y_{\frac{t}{2}-\lceil \frac{d_B+b+1}{2} \rceil}\, y_{\frac{t}{2}+\lfloor \frac{d_B+b-3}{2} \rceil}$ is a directed path. Note that $V(B \cup P)$ is contained in
$$\textstyle \{ y_i: \lfloor \frac{d_B}{2} \rfloor \le i \le \frac{t}{2} - \lceil \frac{d_B+2}{2} \rceil \} \cup
\{ y_i: \frac{t}{2}+ \lfloor \frac{d_B-2}{2} \rfloor \le i \le t-\lceil \frac{d_B}{2} \rceil\}.$$

Since $\frac{t-s+1}{2} \le a \le \frac{s}{3}$, it is not difficult to show that $t-s+1 < s+3$, and hence
$a' < \lceil \frac{d_B}{2} \rceil$.
It follows that $B \cup P$ and $A$ are disjoint, and that $D=A \cup B \cup P$
is an $S$-orthogonal $(\vec{P}_1^{\langle a \rangle},\vec{P}_{t-s-a})$-subdigraph of $\dci(t;S)$. See Figure~\ref{fig:Leo2}.

If $|R|=1$, then $t-s-a=2$, $R=\{ \frac{t}{2} \}$, and $d_B+b-1=\frac{t}{2}-2$. Let $P= y_{\frac{t}{2}+\lfloor \frac{t}{4} \rfloor} \,
y_{\frac{t}{2}-\lceil \frac{t}{4} \rceil} \,
y_{\frac{t}{2}+\lfloor \frac{t}{4} \rfloor}-2$, so that $P$ is a directed 2-path disjoint from $A$ and $B$, and $D=A \cup B \cup P$
is an $S$-orthogonal $(\vec{P}_1^{\langle a \rangle},\vec{P}_{2})$-subdigraph of $\dci(t;S)$.

\smallskip

{\sc Subcase 2.4}: $a=t-s$. This case is similar to Subcase 2.3, so we only highlight the differences. Define $A$ exactly as in Subcase 2.3, but let
$$\textstyle R'=S-T=\{ \pm d_B, \pm (d_B+2),\ldots,\pm (\frac{t}{2}-2), \frac{t}{2} \}.$$
For $d \in R'$ such that $d>0$, define  directed 1-paths
$$\textstyle B_{-d}=y_{\lfloor \frac{d}{2} \rfloor} \, y_{-\lceil \frac{d}{2} \rceil}   \qquad \mbox{ and }\qquad
B_{d}= y_{\frac{t}{2}-\lceil \frac{d}{2} \rceil} \, y_{\frac{t}{2}+\lfloor \frac{d}{2} \rfloor},$$
so that the arcs in $B_{-d}$ and $B_{d}$ are of difference $-d$ and $d$, respectively.
Then $B= \bigcup_{i \in R'} B_i$ is an $R'$-orthogonal $(\vec{P}_1^{\langle a-2a' \rangle})$-subdigraph of $\dci(t; R')$.  Its vertex set is contained in
$$\textstyle \{ y_i: \lfloor \frac{d_B}{2} \rfloor \le i \le \frac{t}{2} - \lceil \frac{d_B}{2} \rceil \} \cup
\{ y_i: \frac{t}{2}+ \lfloor \frac{d_B}{2} \rfloor \le i \le t-\lceil \frac{d_B}{2} \rceil\},$$
and as before, it can be shown that $A$ and $B$ are disjoint. Hence $D=A \cup B$
is an $S$-orthogonal $(\vec{P}_1^{\langle a \rangle})$-subdigraph of $\dci(t;S)$.


\bigskip

{\sc Case 3}: $t$ and $s$ are both odd. Then $S=\{ \pm \frac{s+1}{2}, \pm \frac{s+3}{2}, \ldots, \pm \frac{t-1}{2}\}.$
Define directed 1-paths
$$\textstyle  A_{i}=y_{i}\, y_{\frac{t-1}{2}-i} \qquad \mbox{ for } i=1,2,\ldots,\frac{a-1}{2},$$
and
$$\textstyle  A_{i}=y_{i}\, y_{\frac{t-1}{2}-i+1}\qquad \mbox{ for } i=-\frac{a-1}{2},-\frac{a-3}{2},\ldots,0.$$
Observe that the arc in $A_i$ is of difference $\frac{t-1}{2}-2i$ if $i>0$, and of difference $-(\frac{t-1}{2}+2i)$ if $i \le 0$.
Thus the $A_{i}$, for $i \in \{ 0, \pm 1, \ldots,\pm \frac{a-1}{2} \}$,
jointly contain exactly one arc of each difference in
$$\textstyle T=\{ - \frac{t-1}{2}, \pm (\frac{t-1}{2}-2),\ldots, \pm (\frac{t-1}{2}-(a-1)) \}.$$
Moreover,
the $A_i$ are pairwise disjoint, and $A= \bigcup_{i=-\frac{a-1}{2}}^{\frac{a-1}{2}} A_{i}$  is a $T$-orthogonal $(\vec{P}_1^{\langle a \rangle})$-subdigraph of $\dci(t; T)$. Its vertex set is contained in
$$\textstyle \{ y_i: -\frac{a-1}{2} \le i \le \frac{a-1}{2} \} \cup
\{ y_i: \frac{t-1}{2}-\frac{a-1}{2} \le i \le \frac{t-1}{2}+\frac{a+1}{2} \}.$$

\smallskip

\begin{figure}[t]
\centerline{\includegraphics[scale=0.6]{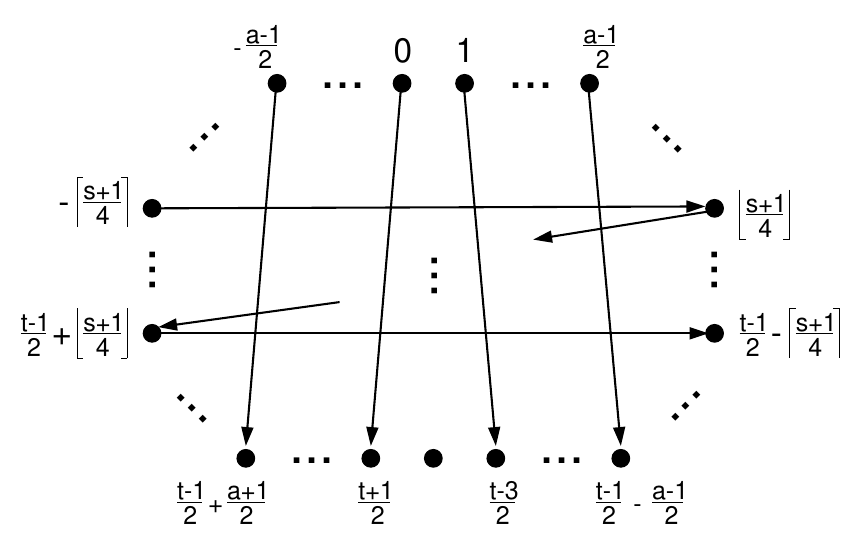}}
\caption{Lemma~\ref{lem:long2}, Subcase 3.1. (Only the subscripts of the vertices are specified.)}
\label{fig:Loo1}
\end{figure}

{\sc Subcase 3.1}: $a \le \frac{t-s-2}{2}$ and $(s,a) \not\in \{ (9,3),(5,1),(3,1) \}$. Note that $\frac{t-1}{2}-(a-1) \ge \frac{s+3}{2}$, so $\frac{s+1}{2} \not\in T$.
Let $R=S-T$. As $\frac{t-1}{2} \in R$, the directed path $P=P(R,y_{- \lceil \frac{s+1}{4} \rceil})$ is well defined, and with the notation $(\star)$, we have $d_1=\frac{s+1}{2}$. Its vertex set is contained in
$$\textstyle \{ y_i: \lfloor \frac{s+1}{4} \rfloor \le i \le \frac{t-1}{2}- \lceil \frac{s+1}{4} \rceil  \} \cup
\{ y_i: \frac{t-1}{2}+\lfloor \frac{s+1}{4} \rfloor \le i \le t- \lceil \frac{s+1}{4} \rceil  \}.$$

Since $a \le \frac{s}{3}$ and $(s,a) \not\in \{ (9,3),(5,1),(3,1) \}$, we have that $\frac{a+1}{2}< \lfloor \frac{s+1}{4} \rfloor$, implying that $P$ is disjoint from $A$. It follows that $D=A \cup P$
is an $S$-orthogonal $(\vec{P}_1^{\langle a \rangle},\vec{P}_{t-s-a})$-subdigraph of $\dci(t;S)$. See Figure~\ref{fig:Loo1}.

\smallskip

{\sc Subcase 3.2}: $(s,a) \in \{ (9,3),(5,1),(3,1) \}$. Use $P$ as in Subcase 3.1, but replace $A$ as follows.
If $(s,a)=(9,3)$, let $A$ be the union of $A_{-1}=y_{-2} \, y_{\frac{t+1}{2}}$,
$A_{0}=y_{-1} \, y_{\frac{t-1}{2}}$, and $A_{1}=y_{0} \, y_{\frac{t-5}{2}}$.
If $(s,a) \in \{ (3,1),(5,1) \}$, let $A=y_{\frac{t-1}{2}}\, y_{0}$.
Then complete the proof as in Subcase 3.1.

\smallskip

\begin{figure}[t]
\centerline{\includegraphics[scale=0.6]{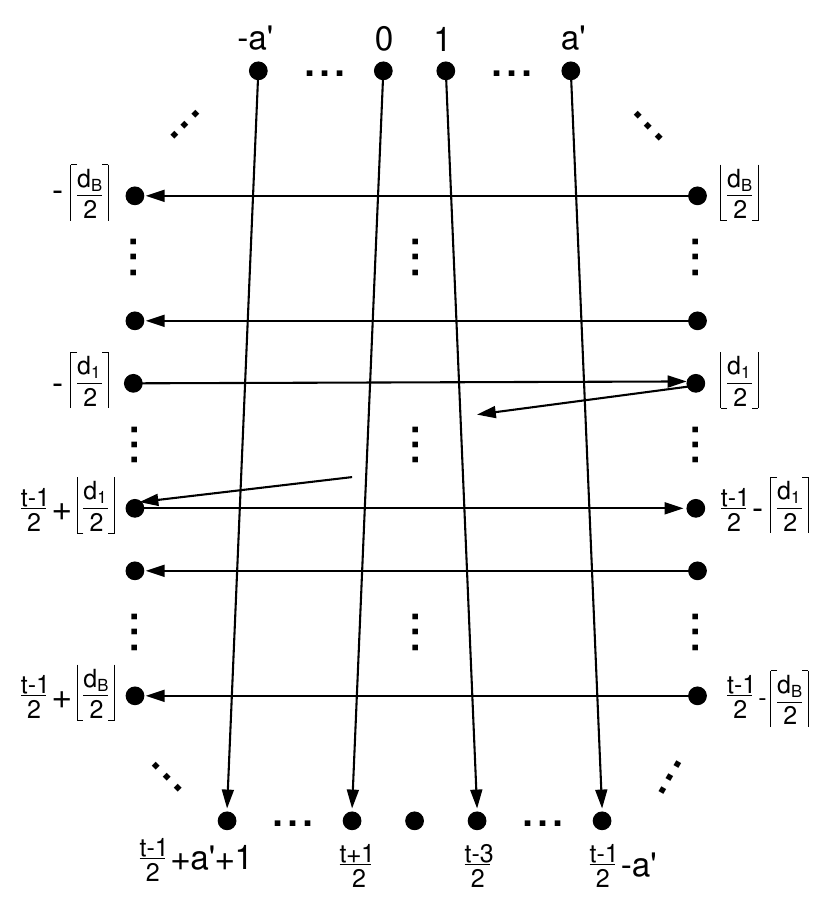}}
\caption{Lemma~\ref{lem:long2}, Subcase 3.3, and Subcase 4.3 with $B_0$ not shown. (Only the subscripts of the vertices are specified.)}
\label{fig:Loo2}
\end{figure}

{\sc Subcase 3.3}: $\frac{t-s}{2} \le a < t-s-1$. Let $a'=\lfloor \frac{t-s-2}{4} \rfloor$ and $d_A=\frac{t-1}{2}-2a'$. Observe that $d_A=\frac{s+1}{2}$ if $t \equiv s+2 \pmod{4}$, and $d_A=\frac{s+3}{2}$ otherwise. Let
$$\textstyle T'=\{ -\frac{t-1}{2}, \pm (\frac{t-1}{2}-2),\ldots, \pm  d_A \},$$ so that $|T'|=2a'+1$.

Let $A'=\bigcup_{i=-a'}^{a'} A_i$. Then $A'$ is a $T'$-orthogonal $(\vec{P}_1^{\langle 2a'+1 \rangle})$-subdigraph of $\dci(t;T')$. Its vertex set is contained in
$$\textstyle V(A')=\{ y_i: -a' \le i \le a' \} \cup
\{ y_i: \frac{t-1}{2}-a' \le i \le \frac{t-1}{2}+a'+1 \}.$$

Let $b=a-2a'-1$, and note that $b$ is even, $b \ge 0$. Let $d_B=\frac{s+3}{2}$ if $t \equiv s+2 \pmod{4}$, and $d_B=\frac{s+1}{2}$ otherwise, so that $\{ d_A,d_B \} = \{ \frac{s+1}{2}, \frac{s+3}{2} \}$ in both cases.
Let
$$R'=\{ \pm d_B, \pm (d_B+2), \ldots, \pm (d_B+b-2) \},$$ so that $|R'|=b$ and $R' \cap T' =\emptyset$. Since $a$ is odd, note that $a<t-s-2$, and hence $d_B+b < \frac{t-1}{2}$.

For $d \in R'$ such that $d>0$, define  directed 1-paths
$$\textstyle B_{-d}=y_{\lfloor \frac{d}{2} \rfloor} \, y_{-\lceil \frac{d}{2} \rceil}   \qquad \mbox{ and }\qquad
B_{d}= y_{\frac{t-1}{2}-\lceil \frac{d}{2} \rceil} \, y_{\frac{t-1}{2}+\lfloor \frac{d}{2} \rfloor}.$$
Observe that the arcs in $B_{-d}$ and $B_{d}$ are of difference $-d$ and $d$, respectively. Thus, the directed 1-paths in $\{ B_i: i \in R' \}$  jointly contain exactly one arc of each difference in $R'$. Moreover, the digraphs $B_i$ are pairwise disjoint. Thus
$B= \bigcup_{i \in R'} B_i$ is an $R'$-orthogonal $(\vec{P}_1^{\langle b \rangle})$-subdigraph of $\dci(t; R')$.

Let $R=S-T'-R'$, and as $\frac{t-1}{2} \in R$ and $|R|=t-s-a \ge 3$, observe that the directed path $P=P(R,y_{-\lceil \frac{d_1}{2} \rceil})$ is well defined, with $d_1=d_B+b$, assuming Notation $(\star)$. Moreover,
$P$ is disjoint from $B$, and $V(B \cup P)$ is contained in
$$\textstyle \{ y_i: \lfloor \frac{d_B}{2} \rfloor \le i \le \frac{t-1}{2} - \lceil \frac{d_B}{2} \rceil \} \cup
\{ y_i: \frac{t-1}{2}+ \lfloor \frac{d_B}{2} \rfloor \le i \le t-\lceil \frac{d_B}{2} \rceil\}.$$

Since $\frac{t-s}{2} \le a \le \frac{s}{3}$, we have $t-s \le \frac{2}{3}s<s-1$, and hence
$a'+1 < \lfloor \frac{d_B}{2} \rfloor$, unless $s=3$.
It follows that $B \cup P$ and $A$ are disjoint, and that $D=A \cup B \cup P$ is an $S$-orthogonal $(\vec{P}_1^{\langle a \rangle},\vec{P}_{t-s-a})$-subdigraph of $\dci(t;S)$. See Figure~\ref{fig:Loo2}.

If $s=3$, then $a=1$, $t=5$, and $D=y_0 \, y_{-2} \cup y_{-1} \, y_1$ is the desired $S$-orthogonal $(\vec{P}_1^{\langle 2 \rangle})$-subdigraph of $\dci(t;S)$.

\smallskip

\begin{figure}[t]
\centerline{\includegraphics[scale=0.6]{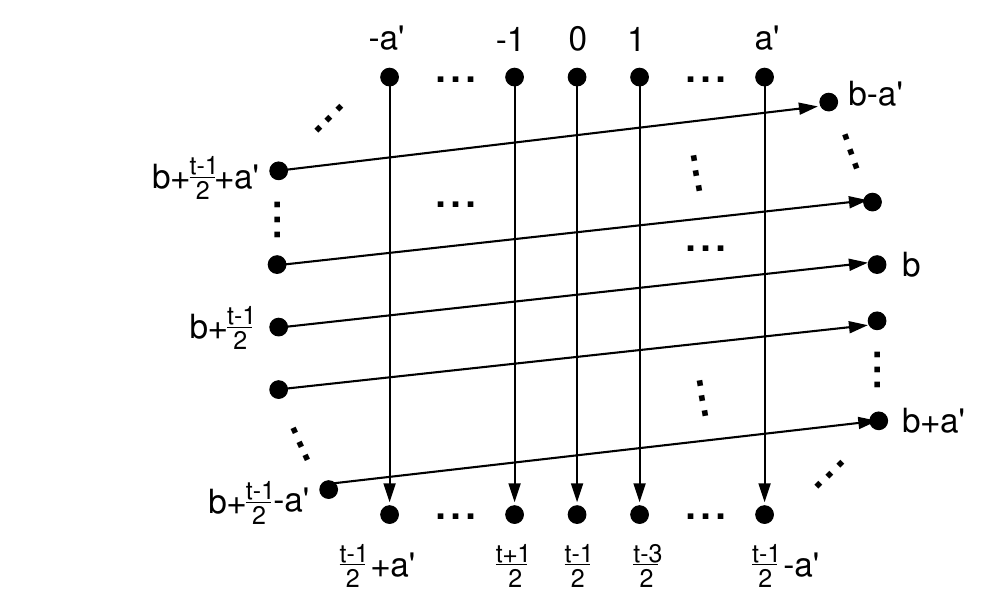}}
\caption{Lemma~\ref{lem:long2}, Subcase 3.4, and Subcase 4.4 with $B_0$ not shown. (Only the subscripts of the vertices are specified.)}
\label{fig:Loo3}
\end{figure}

{\sc Subcase 3.4}: $a=t-s-1$. In this case, we need to construct an $S$-orthogonal $(\vec{P}_1^{\langle a+1 \rangle})$-subdigraph of $\dci(t;S)$. Let $a'=\lceil \frac{t-s-2}{4} \rceil$ and $b=2a'+1$.

For  $i = -a',-a'+1, \ldots, a'$, define directed 1-paths
$$\textstyle B_i=y_i \, y_{\frac{t-1}{2}-i} \qquad \mbox{ and } \qquad
C_i=y_{b+\frac{t-1}{2}-i} \, y_{b+i},$$
and let
$$B=\bigcup_{i=-a'}^{a'} B_i \qquad \mbox{ and } \qquad C=\bigcup_{i=-a'}^{a'} C_i.$$
See Figure~\ref{fig:Loo3}. Observe that the arc in $B_i$ is of difference $\frac{t-1}{2}-2i$ if $i \ge 0$, and of difference $-(\frac{t-1}{2}+2i+1)$ if $i < 0$. Thus the $B_{i}$, for $i \in \{ 0, \pm 1, \ldots,\pm a' \}$,
jointly contain exactly one arc of each difference in
$$\textstyle T_B=\{ \frac{t-1}{2}, -(\frac{t-1}{2}-1),\frac{t-1}{2}-2, \ldots, -(\frac{t-1}{2}-2a'+1), \frac{t-1}{2}-2a' \}.$$
Moreover, the $B_i$ are pairwise disjoint.

Similarly, the arc in $C_i$ is of difference $-(\frac{t-1}{2}-2i)$ if $i \ge 0$, and of difference $\frac{t-1}{2}+2i+1$ if $i < 0$. Thus the $C_{i}$, for $i \in \{ 0, \pm 1, \ldots,\pm a' \}$,
jointly contain exactly one arc of each difference in
$$\textstyle T_C=\{ -\frac{t-1}{2}, \frac{t-1}{2}-1,-(\frac{t-1}{2}-2), \ldots, \frac{t-1}{2}-2a'+1, -(\frac{t-1}{2}-2a') \},$$
and they are also pairwise disjoint.

Furthermore, observe that
$$\textstyle V(B)=\{ y_i: -a' \le i \le a' \} \cup
\{ y_i: \frac{t-1}{2}-a' \le i \le \frac{t-1}{2}+a' \}$$
and
$$\textstyle V(C)=\{ y_i: b-a' \le i \le b+a' \} \cup
\{ y_i: b+\frac{t-1}{2}-a' \le i \le b+\frac{t-1}{2}+a' \}.$$
We claim that $B$ and $C$ are disjoint. Since $b=2a'+1$,  it suffices to show that $\frac{t-1}{2}-a'>b+a'$, that is, that $4a'< \frac{t-3}{2}$. As $a' \le \frac{t-s}{4}$, it further suffices to show $t-s < \frac{t-3}{2}$, which is equivalent to $t-s-1 < s-4$. Since $a=t-s-1 \le \frac{s}{3}$ by assumption, the desired inequality holds for all $s \ge 7$.

Let
$$D= \bigcup_{i=-a'}^{a'} (B_i \cup C_i)\qquad \mbox{ and } \qquad D= \bigcup_{i=-a'}^{a'-1} (B_i \cup C_i)$$
when $t-s \equiv 2 \pmod{4}$ and $t-s \equiv 0 \pmod{4}$, respectively. Note that in either case, $D$ contains exactly one arc of each difference in $S$. Thus, if $s \ge 7$, then $D$ is an $S$-orthogonal $(\vec{P}_1^{\langle a+1 \rangle})$-subdigraph of $\dci(t;S)$ by the above observations. If $s \in \{ 3,5 \}$, however, this statement can easily be verified directly.

\bigskip

{\sc Case 4}: $t$ is odd and $s$ is even. Then $S=\{ 1, \pm \frac{s+2}{2}, \pm \frac{s+4}{2}, \ldots, \pm \frac{t-1}{2}\}.$ In particular, $S=\{ 1 \}$ if $t-s=1$. In this case, the required $S$-orthogonal $(\vec{P}_1)$-subdigraph of $\dci(t;S)$ is easy to find. Hence we may assume $t-s \ge 3$ and $\pm \frac{t-1}{2} \in S$.

If $a \ge 2$, define directed 1-paths
$$\textstyle A_{i}=y_{i}\, y_{\frac{t-1}{2}-i} \qquad \mbox{ for } i=1,2,\ldots,\frac{a-2}{2},$$
and
$$\textstyle A_{i}=y_{i}\, y_{\frac{t-1}{2}-i+1} \qquad \mbox{ for } i=-\frac{a-2}{2},-\frac{a-4}{2}, \ldots,0.$$
Observe that the arc in $A_i$ is of difference $\frac{t-1}{2}-2i$ if $i>0$, and of difference $-(\frac{t-1}{2}+2i)$ if $i \le 0$.
Thus the $A_{i}$, for $i \in \{ 0, \pm 1, \ldots,\pm \frac{a-2}{2} \}$,
jointly contain exactly one arc of each difference in
$$\textstyle T=\{ - \frac{t-1}{2}, \pm (\frac{t-1}{2}-2),\ldots, \pm (\frac{t-1}{2}-(a-2)) \}.$$
 Moreover,
the $A_i$ are pairwise disjoint, and $A= \bigcup_{i=-\frac{a-2}{2}}^{\frac{a-2}{2}} A_{i}$  is a $T$-orthogonal $(\vec{P}_1^{\langle a-1 \rangle})$-subdigraph of $\dci(t; T)$. Its vertex set is contained in
$$\textstyle \{ y_i: -\frac{a-2}{2} \le i \le \frac{a-2}{2} \} \cup
\{ y_i: \frac{t-1}{2}-\frac{a-2}{2} \le i \le \frac{t-1}{2}+\frac{a}{2} \}.$$

\smallskip

\begin{figure}[t]
\centerline{\includegraphics[scale=0.6]{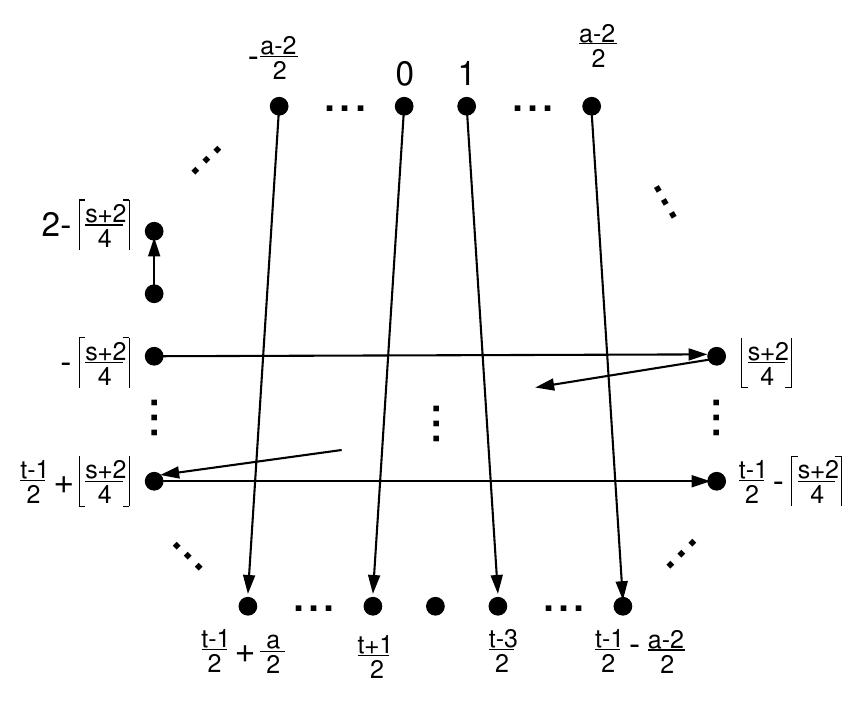}}
\caption{Lemma~\ref{lem:long2}, Subcase 4.1. (Only the subscripts of the vertices are specified.)}
\label{fig:Loe1}
\end{figure}

{\sc Subcase 4.1}: $2 \le a \le \frac{t-s-1}{2}$. Note that $\frac{t-1}{2}-(a-2) \ge \frac{s+4}{2}$, so $\frac{s+2}{2} \not\in T$.
Let $R=S-T-\{ 1 \}$. As $\frac{t-1}{2} \in R$, the directed path $P=P(R,y_{- \lceil \frac{s+2}{4} \rceil})$ is well defined, and with Notation $(\star)$, we have $d_1=\frac{s+2}{2}$. Note that $V(P)$ is contained in
$$\textstyle \{ y_i: \lfloor \frac{s+2}{4} \rfloor \le i \le \frac{t-1}{2}- \lceil \frac{s+2}{4} \rceil  \} \cup
\{ y_i: \frac{t-1}{2}+\lfloor \frac{s+2}{4} \rfloor \le i \le t- \lceil \frac{s+2}{4} \rceil  \}.$$
Additionally, define the directed 1-path $B_0=y_{1-\lceil \frac{s+2}{4} \rceil} \, y_{2-\lceil \frac{s+2}{4} \rceil}$, and observe that $B_0$ is disjoint from $P$.

If $(s,a) \ne (6,2)$, since $2 \le a \le \frac{s}{3}$, we have that $\frac{a-2}{2}-1< \lceil \frac{s+2}{4} \rceil -2$, implying that $B_0 \cup P$ is disjoint from $A$. It follows that $D=A \cup B_0 \cup P$
is an $S$-orthogonal $(\vec{P}_1^{\langle a \rangle},\vec{P}_{t-s-a})$-subdigraph of $\dci(t;S)$. See Figure~\ref{fig:Loe1}.

If $(s,a) = (6,2)$, replace $A$ with the directed 1-path $A'=y_{\frac{t+1}{2}} \, y_1$, and complete as above.

\smallskip

{\sc Subcase 4.2}: $a=0$. Let $R=S-\{ 1, -\frac{t-1}{2} \}$. Then $P=P(R,y_{- \lceil \frac{s+2}{4} \rceil})$ is well defined, with $d_1=\frac{s+2}{2}$, and $P \, y_{\frac{t-1}{2}-\lceil \frac{s+2}{4} \rceil} \, y_{\frac{t-1}{2}-\lceil \frac{s+2}{4} \rceil+1} \, y_{1- \lceil \frac{s+2}{4} \rceil}$ is the required $S$-orthogonal directed $(t-s)$-path in $\dci(t;S)$

\smallskip

{\sc Subcase 4.3}: $\frac{t-s+1}{2} \le a < t-s-1$. Let $a'=\lfloor \frac{t-s-3}{4} \rfloor$ and $d_A=\frac{t-1}{2}-2a'$. Observe that $d_A=\frac{s+2}{2}$ if $t \equiv s+3 \pmod{4}$, and $d_A=\frac{s+4}{2}$ otherwise. Let
$$\textstyle T'=\{ -\frac{t-1}{2}, \pm (\frac{t-1}{2}-2),\ldots, \pm  d_A \},$$
so that $|T'|=2a'+1$.
Let $A'=\bigcup_{i=-a'}^{a'} A_i$. Then $A'$ is a $T'$-orthogonal $(\vec{P}_1^{\langle 2a'+1 \rangle})$-subdigraph of $\dci(t;T')$. Its vertex set is contained in
$$\textstyle V(A')=\{ y_i: -a' \le i \le a' \} \cup
\{ y_i: \frac{t-1}{2}-a' \le i \le \frac{t-1}{2}+a'+1 \}.$$

Let $b=a-2a'-2$, and note that $b$ is even, $b \ge 0$. Let $d_B=\frac{s+4}{2}$ if $t \equiv s+3 \pmod{4}$, and $d_B=\frac{s+2}{2}$ otherwise, so that $\{ d_A,d_B \} = \{ \frac{s+2}{2}, \frac{s+4}{2} \}$ in both cases.
Let
$$\textstyle R'=\{ \pm d_B, \pm (d_B+2), \ldots, \pm (d_B+b-2) \},$$ so that $|R'|=b$ and $R' \cap T' =\emptyset$. Note that since $a<t-s-2$, we have that $d_B+b < \frac{t-1}{2}$.

For $d \in R'$ such that $d>0$, define  directed 1-paths
$$\textstyle B_{-d}=y_{\lfloor \frac{d}{2} \rfloor} \, y_{-\lceil \frac{d}{2} \rceil}   \qquad \mbox{ and }\qquad
B_{d}= y_{\frac{t-1}{2}-\lceil \frac{d}{2} \rceil} \, y_{\frac{t-1}{2}+\lfloor \frac{d}{2} \rfloor}.$$
Observe that the arcs in $B_{-d}$ and $B_{d}$ are of difference $-d$ and $d$, respectively. Thus, the directed 1-paths in $\{ B_i: i \in R' \}$  jointly contain exactly one arc of each difference in $R'$. Moreover, the digraphs $B_i$ are pairwise disjoint. Thus
$B= \bigcup_{i \in R'} B_i$ is an $R'$-orthogonal $(\vec{P}_1^{\langle b \rangle})$-subdigraph of $\dci(t; R')$.

Let $R=S-T'-R'-\{ 1 \}$, and as $\frac{t-1}{2} \in R$ and $|R|=t-s-a \ge 3$, observe that the directed path $P=P(R,y_{-\lceil \frac{d_1}{2} \rceil})$ is well defined, with $d_1=d_B+b$ using Notation $(\star)$. Moreover,
$P$ is disjoint from $B$, and $V(B \cup P)$ is contained in
$$\textstyle \{ y_i: \lfloor \frac{d_B}{2} \rfloor \le i \le \frac{t-1}{2} - \lceil \frac{d_B}{2} \rceil \} \cup
\{ y_i: \frac{t-1}{2}+ \lfloor \frac{d_B}{2} \rfloor \le i \le t-\lceil \frac{d_B}{2} \rceil\}.$$
Additionally, define the directed 1-path $B_0=y_{1-\lceil \frac{d_B}{2} \rceil} \, y_{2-\lceil \frac{d_B}{2} \rceil}$, and observe that $B_0$ is disjoint from $B \cup P$.

Since $\frac{t-s+1}{2} \le a \le \frac{s}{3}$, we have $t-s+1\le \frac{2}{3}s<s$, and hence
$a'+2 < \lceil \frac{d_B}{2} \rceil$.
It follows that $B \cup B_0 \cup P$ and $A$ are disjoint, and that $D=A \cup B \cup B_0 \cup P$ is an $S$-orthogonal $(\vec{P}_1^{\langle a \rangle},\vec{P}_{t-s-a})$-subdigraph of $\dci(t;S)$. Figure~\ref{fig:Loo2} shows $A \cup B \cup P$ only.

{\sc Subcase 4.4}: $a=t-s-1$. This subcase is very similar to Subcase 3.4, so we only highlight the differences. We are constructing an $S$-orthogonal $(\vec{P}_1^{\langle a+1 \rangle})$-subdigraph of $\dci(t;S)$. Let $a'=\lceil \frac{t-s-3}{4} \rceil$ and $b=2a'+1$.

For  $i = -a',-a'+1, \ldots, a'$, define directed 1-paths $B_i$ and $C_i$, and their unions $B$ and $C$, respectively, exactly as in Subcase 3.4 (see Figure~\ref{fig:Loo3}). In addition, let $D_0$ be the directed 1-path $D_0= y_{-a'-2} \, y_{-a'-1}$.

Let
$$D= D_0 \cup \bigcup_{i=-a'}^{a'} (B_i \cup C_i)\qquad \mbox{ and } \qquad D= D_0 \cup \bigcup_{i=-a'}^{a'-1} (B_i \cup C_i)$$
when $t-s \equiv 3 \pmod{4}$ and $t-s \equiv 1 \pmod{4}$, respectively. Note that in both cases, $D$ contains exactly one arc of each difference in $S$. To see that $D$ is a disjoint union of directed 1-paths, it suffices to verify that $b+\frac{t-1}{2}+a'< t-a'-2$, that is, that $4a'< \frac{t-5}{2}$. As $a' \le \frac{t-s-1}{4}$, it further suffices to show that $t-s-1 < \frac{t-5}{2}$, which holds for $s \ge 7$ as seen in Subcase 3.4. For $s <7$, we have either $a=0$, or else $s=6$, $a=2$, and $t=9$. In each of these cases, the above construction can easily be verified directly.

\bigskip
This completes all cases.
\end{proof}


In the next lemma, we provide constructions for the special cases when either Lemma~\ref{lem:long1} or Lemma~\ref{lem:long2} does not apply.

\begin{lemma}\label{lem:long-special}
Let $t$ be an even integer.
\begin{enumerate}[(a)]
\item If $t\ge 4$, then OP$^\ast(3,t)$ has a solution.
\item If $t\ge 8$, then OP$^\ast(2,2,t)$ has a solution.
\end{enumerate}
\end{lemma}

\begin{proof} Since OP$^\ast(3)$ and OP$^\ast(2,2)$ have solutions by Theorem~\ref{thm:Kn*}, it suffices to show that Condition (2) of Corollary ~\ref{cor:main} is satisfied.

\begin{enumerate}[(a)]
\item Using Corollary ~\ref{cor:main}, we have $\ell=1$, $k=2$, $s=m_1=3$, and $m_2=t$. We take $s_1=t_1=1$, $s_2=2$, and $t_2=t-4$, and it then suffices to find $S \subset \ZZ_t^\ast$ such that $\dci(t;\ZZ_t^\ast-S)$ admits a $\vec{C}_t$-factorization and $\dci(t;S)$ admits an $S$-orthogonal $(\vec{P}_1,\vec{P}_{t-4})$-subdigraph.

First, let $t \ge 8$, and assume $t \equiv 0 \pmod{4}$. Let $D=\{ -1, \frac{t}{2}-1 \}$. Since $\gcd(t,\frac{t}{2}-1)=1$ and $\frac{t}{2}-1>1$, it is easy to see that $\dci(t;D)$ admits a $\vec{C}_t$-factorization with two $\vec{C}_t$-factors.

Let $S=\ZZ_t^\ast-D=\{ 1, \pm 2, \pm 3, \ldots, \pm (\frac{t}{2}-2), -(\frac{t}{2}-1), \frac{t}{2} \}$, and let the vertex set of $\dci(t;S)$ be $\{ y_i: i \in \ZZ_t \}$. Construct the following directed paths in $\dci(t;S)$:
$$ D_1= y_0 \, y_2 \, y_{-1} \, y_{3} \, y_{-2} \, \ldots \, y_{-(\frac{t}{4}-2)} \, y_{\frac{t}{4}} \, y_{-(\frac{t}{4}-1)} \, y_{\frac{t}{4}+3} \, y_{-\frac{t}{4}} \, y_{\frac{t}{4}+4} \, \ldots \, y_{-(\frac{t}{2}-3)} \, y_{\frac{t}{2}+1} \, y_1$$
and
$D_2=y_{\frac{t}{4}+1} \, y_{\frac{t}{4}+2}.$
Note that the differences of the arcs in $D_1$ are, in order,
$$\textstyle 2, -3, 4, -5, \ldots, \frac{t}{2}-2, -(\frac{t}{2}-1), -(\frac{t}{2}-2), \frac{t}{2}-3, \ldots, -2,\frac{t}{2},$$
while the arc of $D_2$ is of difference 1.
Hence $D_1 \cup D_2$ is an $S$-orthogonal $(\vec{P}_1,\vec{P}_{t-4})$-subdigraph of $\dci(t;S)$.

Assume now that $t \equiv 2 \pmod{4}$, and let $D=\{ \pm (\frac{t}{2}-2) \}$. Since $\gcd(t,\frac{t}{2}-2)=1$, we know that $\dci(t;D)$ admits a $\vec{C}_t$-factorization. Note that  $\frac{t}{2}-2>1$ by the assumption.

Let $S=\ZZ_t^\ast-D=\{ \pm 1, \pm 2,  \ldots, \pm (\frac{t}{2}-3), \pm(\frac{t}{2}-1), \frac{t}{2} \}$, and let the vertex set of $\dci(t;S)$ be $\{ y_i: i \in \ZZ_t \}$. Construct the following directed paths in $\dci(t;S)$:
$$ D_1= y_0 \, y_2 \, y_{-1} \, y_{3} \, y_{-2} \, \ldots \, y_{-\frac{t-10}{4}} \, y_{\frac{t-2}{4}} \, y_{-\frac{t-2}{4}} \,  y_{-\frac{t+2}{4}} \, y_{\frac{t+10}{4}} \, y_{-\frac{t+6}{4}} \, \ldots \, y_{-\frac{t-4}{2}} \, y_{\frac{t}{2}} \, y_{\frac{t}{2}+1} \, y_1$$
and
$D_2=y_{-\frac{t-6}{4}} \,  y_{\frac{t+2}{4}}.$
Note that the differences of the arcs in $D_1$ are, in order,
$$\textstyle 2,-3, 4, -5, \ldots, \frac{t}{2}-3, -(\frac{t}{2}-1), -1, -(\frac{t}{2}-3), \frac{t}{2}-4, \ldots, -2, 1, \frac{t}{2},$$
while $D_2$ contains an arc of difference $\frac{t}{2}-1$. Hence $D_1 \cup D_2$ is an $S$-orthogonal $(\vec{P}_1,\vec{P}_{t-4})$-subdigraph.

For $t=4$, let $D=\{ \pm 1 \}$ and $S=\{2\}$. For $t=6$, let $D=\{ \pm 1 \}$ and $S=\{ \pm 2, 3 \}$. In both cases, it is easy to verify that $\dci(t;D)$ admits a $\vec{C}_t$-factorization and that $\dci(t;S)$ admits an $S$-orthogonal $(\vec{P}_1,\vec{P}_{t-4})$-subdigraph.

In all cases, by Corollary~\ref{cor:main}, it follows that OP$^\ast(3,t)$ has a solution.

\item Using Corollary ~\ref{cor:main}, we now have $\ell=2$, $k=3$, $m_1=m_2=2$, $s=4$, and $m_3=t$. We take $s_1=s_2=1$, $t_1=t_2=0$, $s_3=2$, and $t_3=t-4$. It then suffices to find $S \subset \ZZ_t^\ast$ such that $\dci(t;\ZZ_t^\ast-S)$ admits a $\vec{C}_t$-factorization and $\dci(t;S)$ admits an $S$-orthogonal $(\vec{P}_{t-4})$-subdigraph.

Recall that $t \ge 8$.
If $t \equiv 0 \pmod{4}$, let $D=\{ -1, \pm (\frac{t}{2}-1) \}$. Since $\gcd(t,\frac{t}{2}-1)=1$ and $\frac{t}{2}-1>1$, it is easy to see that $\dci(t;D)$ admits a $\vec{C}_t$-factorization with three $\vec{C}_t$-factors.

Let $S=\ZZ_t^\ast-D=\{ 1, \pm 2, \pm 3, \ldots, \pm (\frac{t}{2}-2), \frac{t}{2} \}$, and let the vertex set of $\dci(t;S)$ be $\{ y_i: i \in \ZZ_t \}$. Construct the following directed path in $\dci(t;S)$:
$$ P= y_0 \, y_2 \, y_{-1} \, y_{3} \, y_{-2} \, \ldots \, y_{-(\frac{t}{4}-2)} \, y_{\frac{t}{4}} \, y_{-\frac{t}{4}} \, y_{\frac{t}{4}+2} \, y_{-(\frac{t}{4}+1)} \, y_{\frac{t}{4}+3} \, \ldots \, y_{-(\frac{t}{2}-2)} \, y_{\frac{t}{2}} \, y_{\frac{t}{2}+1}.$$
Note that the differences of the arcs in $P$ are, in order,
$$\textstyle 2, -3, 4, -5, \ldots, \frac{t}{2}-2, \frac{t}{2}, -(\frac{t}{2}-2), \frac{t}{2}-3, \ldots, -2,1,$$
so that $P$ is an $S$-orthogonal $(\vec{P}_{t-4})$-subdigraph.

If $t \equiv 2 \pmod{4}$, let $D=\{ -1, \pm (\frac{t}{2}-2) \}$. Since $\gcd(t,\frac{t}{2}-2)=1$ and $\frac{t}{2}-2>1$, we know that $\dci(t;D)$ admits a $\vec{C}_t$-factorization with three $\vec{C}_t$-factors.

Let $S=\ZZ_t^\ast-D=\{ 1, \pm 2,  \ldots, \pm (\frac{t}{2}-3), \pm(\frac{t}{2}-1), \frac{t}{2} \}$, and let the vertex set of $\dci(t;S)$ be $\{ y_i: i \in \ZZ_t \}$. Construct the following directed path in $\dci(t;S)$:
$$ P= y_0 \, y_2 \, y_{-1} \, y_{3} \, y_{-2} \, \ldots \, y_{-\frac{t-10}{4}} \, y_{\frac{t-2}{4}} \, y_{-\frac{t-2}{4}} \,
y_{\frac{t+2}{4}} \, y_{-\frac{t+2}{4}} \,y_{\frac{t+10}{4}} \, y_{-\frac{t+6}{4}} \, \ldots \, y_{-\frac{t-4}{2}} \, y_{\frac{t}{2}} \, y_{\frac{t}{2}+1}.$$
Note that the differences of the arcs in $P$ are, in order,
$$\textstyle 2,-3, 4, -5, \ldots, \frac{t}{2}-3, -(\frac{t}{2}-1), \frac{t}{2}, \frac{t}{2}-1, -(\frac{t}{2}-3), \frac{t}{2}-4, \ldots, -2, 1,$$
so that $P$ is an $S$-orthogonal $(\vec{P}_{t-4})$-subdigraph.

In both cases, by Corollary~\ref{cor:main}, it follows that OP$^\ast(2,2,t)$ has a solution.
\end{enumerate}
\end{proof}


\section{Extending the 2-factor by 2-cycles: \\ technical lemmas}

Analogously to the previous section, we shall now perform the majority of the work required to prove Theorem~\ref{the:main-ext}(2); that is, we present the constructions that, with the help of Corollary~\ref{cor:main}, enable us  to  obtain a solution to OP$^\ast(m_1,\ldots,m_{\ell},2^{\langle b \rangle})$ from
a solution to OP$^\ast(m_1,\ldots,m_{\ell})$  when $b$ is sufficiently large.

\begin{lemma}\label{lem:2-cycles1}
Let $s$ and $t$  be integers such that  $t$ is even and $2 \le s < t$. Furthermore,
let $$\textstyle D=\{ \pm 1, \pm 2, \ldots, \pm \frac{s-1}{2} \}$$
if $s$ is odd, and
$$\textstyle D=\{ \pm 1, \pm 2, \ldots, \pm (\frac{s}{2}-1)\} \cup \{ \frac{t}{2} \}$$
if $s$ is even.
Then $\dci(t;D)$ admits a $\vec{C}_2$-factorization.
\end{lemma}

\begin{proof}
Let $k=\lfloor \frac{s-1}{2} \rfloor$ and  $T=\{ 1,2, \ldots, k\}$. Note that $T$ has a partition
$$\textstyle \P = \left\{ \{ 1 \}, \{ 2,3 \}, \ldots, \{ k-1, k \} \right\}$$
or
$$\textstyle \P = \left\{ \{ 1,2 \}, \{ 3,4 \}, \ldots, \{ k-1, k \} \right\}.$$
In either case, for each $S \in \P$, the circulant graph $\ci(t;S \cup (-S))$ is either a cycle or has a decomposition into Hamilton cycles by Theorem~\ref{the:BerFavMah}. Hence $\ci(t;T \cup (-T))$ admits a $C_t$-factorization, and since $t$ is even, it admits a 1-factorization. It follows that $\dci(t;T \cup (-T))$ admits a $\vec{C}_2$-factorization.

When $s$ is odd, we have $D=T \cup (-T)$, so the result follows directly. When $s$ is even, we observe that $\dci(t;D)=\dci(t;T \cup (-T)) \oplus \dci(t;\{ \frac{t}{2} \})$, and that $\dci(t;\{ \frac{t}{2} \})$ is a $\vec{C}_2$-factor of $\dci(t;D)$. Hence $\dci(t;D)$ admits a $\vec{C}_2$-factorization in both cases.
\end{proof}

\begin{lemma}\label{lem:2-cycles2}
Let $a$, $s$, and $t$ be integers such that $t$ is even, $2 \le s < t$,  $a \le \min \{ \lfloor \frac{s}{3} \rfloor, t-s \}$, and $a \equiv s \pmod{2}$.
Furthermore, let $$\textstyle S=\{ \pm \frac{s+1}{2}, \pm \frac{s+3}{2}, \ldots, \pm (\frac{t}{2}-1)\} \cup \{ \frac{t}{2} \}$$
if $s$ is odd, and
$$\textstyle S=\{ \pm \frac{s}{2}, \pm (\frac{s}{2}+1), \ldots, \pm (\frac{t}{2}-1) \}$$
if $s$ is even.
Then  $\dci(t;S)$ admits an $S$-orthogonal $(\vec{P}_1^{\langle a \rangle},\vec{C}_{2}^{\langle \frac{t-s-a}{2} \rangle})$-subdigraph.
\end{lemma}

\begin{proof} Note that in both cases, $|S|=t-s$. Let the vertex set of the circulant digraph $\dci(t;S)$ be $V=\{ y_i: i \in \ZZ_t \}$.

\bigskip

\begin{figure}[t]
\centerline{\includegraphics[scale=0.6]{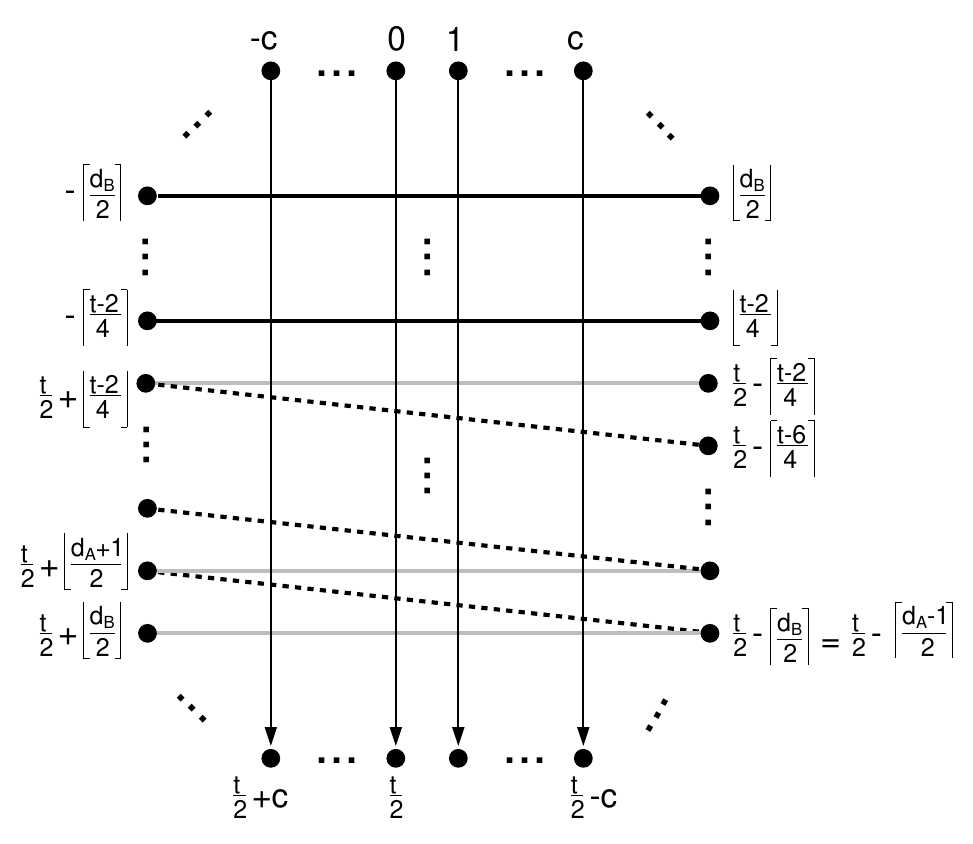}}
\caption{Lemma~\ref{lem:2-cycles2}, Case 1, with $d_A>d_B$ --- digraph $A$ (black arcs) and sets $B_i$ (thick black lines), $E_i$ (grey lines), and $C_i$ (dashed lines). (Only the subscripts of the vertices are specified.)}
\label{fig:So}
\end{figure}

{\sc Case 1:} $s$ is odd. Then $S=\{ \pm \frac{s+1}{2}, \pm \frac{s+3}{2}, \ldots, \pm \frac{t-2}{2}, \frac{t}{2} \}$.
For any $c \in \{ 0, 1, \ldots, \lfloor \frac{t-2}{4} \rfloor \}$ and $i=0, \pm 1, \pm 2, \ldots, \pm c$, define the directed 1-path
$$\textstyle A_{i}=y_{i}\, y_{\frac{t}{2}-i}.$$
Observe that the arc in $A_i$ is of difference $\frac{t}{2}-2i$ if $i \ge 0$, and of difference $-(\frac{t}{2}+2i)$ if $i < 0$.
Thus the $A_{i}$, for $i \in \{ 0, \pm 1, \ldots,\pm c \}$,
jointly contain exactly one arc of each difference in
$$\textstyle T_c=\{ \frac{t}{2}, \pm (\frac{t}{2}-2), \pm (\frac{t}{2}-4),\ldots, \pm (\frac{t}{2}-2c) \}.$$
 Moreover,
the $A_i$ are pairwise disjoint, and $A= \bigcup_{i=-c}^{c} A_{i} $  is a $T_c$-orthogonal $(\vec{P}_1^{\langle 2c+1 \rangle})$-subdigraph of $\dci(t; T_c)$. Its vertex set is
$$\textstyle V(A)=\{ y_i: -c \le i \le c \}  \cup
\{ y_i: \frac{t}{2}-c \le i \le \frac{t}{2}+c \}.$$

Let $d_A=\frac{s+1}{2}$ and $d_B=\frac{s+3}{2}$ if $t \equiv s+1 \pmod{4}$, and $d_A=\frac{s+3}{2}$ and $d_B=\frac{s+1}{2}$ otherwise. Thus $d_A \equiv \frac{t}{2} \pmod{2}$ and $d_B \equiv \frac{t}{2}-1 \pmod{2}$ in both cases.
Let $I=\{ d_B, d_B+2, \ldots, \frac{t}{2}-1 \}$ and $J=\{ d_A, d_A+2, \ldots, \frac{t}{2}-2 \}$, and note that $S=I \du J \du \{ \frac{t}{2} \}$.

For $i \in I$, define the sets
$$B_i= \{ y_{-\lceil \frac{i}{2} \rceil}, y_{\lfloor \frac{i}{2} \rfloor} \}
\qquad \mbox{ and } \qquad
E_i= \{ y_{\frac{t}{2}-\lceil \frac{i}{2} \rceil}, y_{\frac{t}{2}+\lfloor \frac{i}{2} \rfloor} \}.$$
For $i \in J$, define the  set
$$C_i= \{  y_{\frac{t}{2}-\lceil \frac{i-1}{2} \rceil
}, y_{\frac{t}{2}+\lfloor \frac{i+1}{2} \rfloor} \}.$$
See Figure~\ref{fig:So}. Observe that the sets $B_i$, for $i \in I$, are pairwise disjoint, as are the sets $E_i$, for $i \in I$, and the sets $C_i$, for $i \in J$. Moreover,
each of the sets $B_i$, $C_i$, and $E_i$ accommodates arcs of difference $\pm i$. Let $B=\bigcup_{i \in I} B_i$, $C=\bigcup_{j \in J} C_j$, and $E=\bigcup_{i \in I} E_i$, and observe that $B \cap C= \emptyset$, $B \cap E =\emptyset$,
$$\textstyle B \cup C \subseteq \{ y_i:  \lfloor \frac{d_B}{2} \rfloor \le i \le \frac{t}{2}- \lceil \frac{d_A-1}{2} \rceil \} \cup
\{ y_i:  \frac{t}{2}+\lfloor \frac{d_A+1}{2} \rfloor  \le i \le t- \lceil \frac{d_B}{2} \rceil \},$$
and
$$\textstyle B \cup E \subseteq \{ y_i:  \lfloor \frac{d_B}{2} \rfloor \le i \le \frac{t}{2}- \lceil \frac{d_B}{2} \rceil \} \cup
\{ y_i:  \frac{t}{2}+\lfloor \frac{d_B}{2} \rfloor  \le i \le t- \lceil \frac{d_B}{2} \rceil \}.$$

\smallskip

{\sc Subcase 1.1:} $a \le \frac{t-s-1}{2}$. Let $c=\frac{a-1}{2}$ and $A= \bigcup_{i=-c}^{c} A_i$, so $A$ is a
$(\vec{P}_1^{\langle a \rangle})$-subdigraph of $\dci(t; T)$ for $T=\{ \frac{t}{2}, \pm (\frac{t}{2}-2), \pm (\frac{t}{2}-4),\ldots, \pm (\frac{t}{2}-2c)) \}$. Note that, since $a \le \frac{t-s-1}{2}$, we have that $\frac{t}{2}-2c \ge d_A$.

Since $a \le \frac{s}{3}$, it is not difficult to show that $c < \min\{ \lceil \frac{d_A-1}{2} \rceil, \lfloor \frac{d_B}{2} \rfloor\}$. It follows that
$V(A)$ is disjoint from $B \cup C$.

For $i \in I$, let $D_i$ be the directed 2-cycle with vertex set $B_i$, and for $j \in J-T$, let $D_j$ be the directed 2-cycle with vertex set $C_j$. Let $F=\bigcup_{i \in I} D_i \cup \bigcup_{j \in J-T} D_j $. Then $D=A \cup F$ is the desired $S$-orthogonal $(\vec{P}_1^{\langle a \rangle},\vec{C}_{2}^{\langle \frac{t-s-a}{2} \rangle})$-subdigraph of  $\dci(t;S)$.

\smallskip

{\sc Subcase 1.2:} $a \ge \frac{t-s+1}{2}$. Let $c=\lfloor \frac{t-s-1}{4} \rfloor$, and
$A= \bigcup_{i=-c}^{c} A_i$, so $A$ is a
$(\vec{P}_1^{\langle 2c+1 \rangle})$-subdigraph of $\dci(t; T)$ for $T=\{ \frac{t}{2}, \pm (\frac{t}{2}-2), \pm (\frac{t}{2}-4),\ldots, \pm d_A \}$.

As $\frac{t-s+1}{2} \le a \le \frac{s}{3}$, we have that $t-s+1 \le \frac{2}{3}s<s$, and hence
$c < \lfloor \frac{d_B}{2} \rfloor$. It follows that
$V(A)$ is disjoint from $B \cup E$.

Let $b=a-(2c+1)$, and observe that $b$ is even, $b \ge 0$. Let $I' \subseteq I$ such that $|I'|=\frac{b}{2}$.
For $i \in I-I'$, let $D_i$ be the directed 2-cycle with vertex set $B_i$, and for $j \in I'$, let $P_j$ and $P_j'$ be the disjoint directed 1-paths with vertex sets in $B_j$ and $E_j$ that contain arcs of difference $j$ and $-j$, respectively. Let $F=\bigcup_{i \in I-I'} D_i \cup \bigcup_{j \in I'} (P_j \cup P_j')$. Then $D=A \cup F$ is the desired $S$-orthogonal $(\vec{P}_1^{\langle a \rangle},\vec{C}_{2}^{\langle \frac{t-s-a}{2} \rangle})$-subdigraph of  $\dci(t;S)$.

\bigskip

\begin{figure}[t]
\centerline{\includegraphics[scale=0.6]{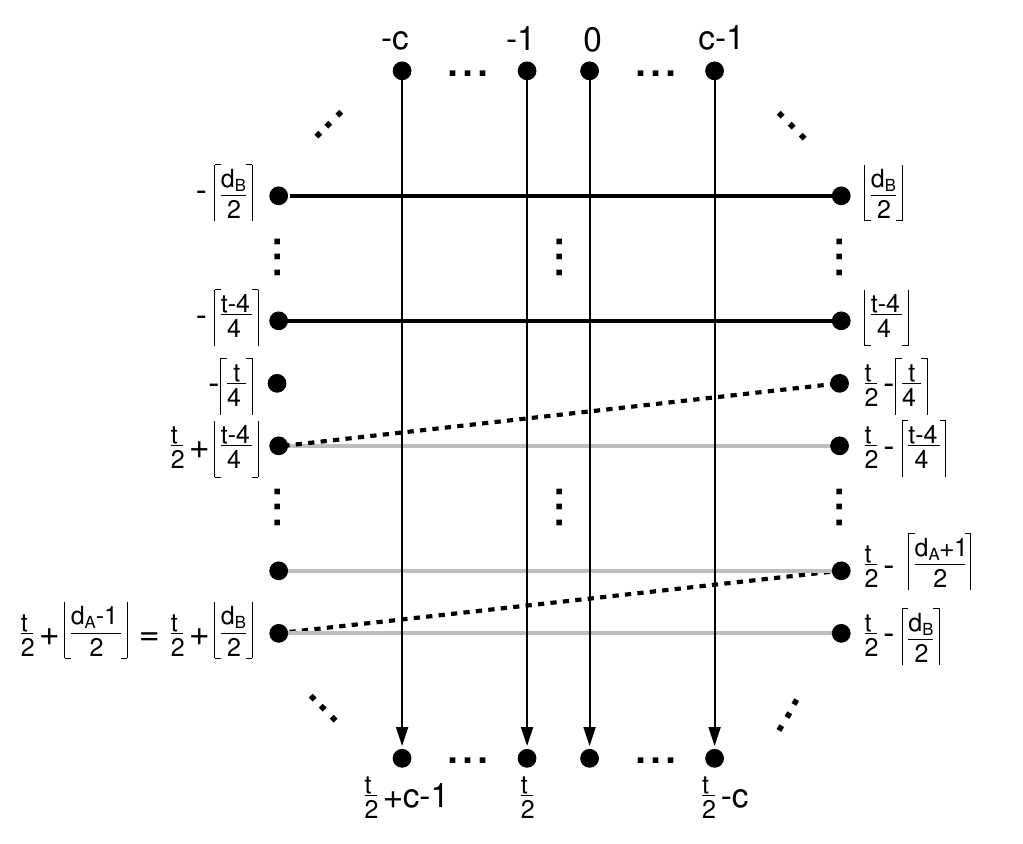}}
\caption{Lemma~\ref{lem:2-cycles2}, Case 2, with $d_A>d_B$ --- digraph $A$ (black arcs) and sets $B_i$ (thick black lines), $E_i$ (grey lines), and $C_i$ (dashed lines). (Only the subscripts of the vertices are specified.)}
\label{fig:Se}
\end{figure}

{\sc Case 2:} $s$ is even. Then $S=\{ \pm \frac{s}{2}, \pm \frac{s+2}{2}, \ldots, \pm \frac{t-2}{2}\}$.
For any $c \in \{ 1,2, \ldots, \lfloor \frac{t}{4} \rfloor \}$ and $i=-c, -c+1, \ldots,c-1$, define the directed 1-path
$$\textstyle A_{i}=y_{i}\, y_{\frac{t}{2}-i-1}.$$
Observe that the arc in $A_i$ is of difference $\frac{t}{2}-(2i+1)$ if $i \ge 0$, and of difference $-(\frac{t}{2}+(2i+1))$ if $i < 0$.
Thus the $A_{i}$, for $i \in \{ -c, -c+1, \ldots,c-1 \}$,
jointly contain exactly one arc of each difference in
$$\textstyle T_c=\{ \pm(\frac{t}{2}-1), \pm (\frac{t}{2}-3), \ldots, \pm (\frac{t}{2}-(2c-1)) \}.$$
 Moreover,
the $A_i$ are pairwise disjoint, and $A= \bigcup_{i=-c}^{c-1} A_{i} $  is a $T_c$-orthogonal $(\vec{P}_1^{\langle 2c \rangle})$-subdigraph of $\dci(t; T_c)$. Its vertex set is
$$\textstyle V(A)=\{ y_i: -c \le i \le c-1 \} \cup
\{ y_i: \frac{t}{2}-c \le i \le \frac{t}{2}+c-1 \}.$$

Let $d_A=\frac{s}{2}$ and $d_B=\frac{s+2}{2}$ if $t \equiv s+2 \pmod{4}$, and $d_A=\frac{s+2}{2}$ and $d_B=\frac{s}{2}$ otherwise.
Note that in both cases, $d_A \equiv \frac{t}{2}-1 \pmod{2}$ and $d_B \equiv \frac{t}{2} \pmod{2}$.
Let $I=\{ d_B, d_B+2, \ldots, \frac{t}{2}-2 \}$ and $J=\{ d_A, d_A+2, \ldots, \frac{t}{2}-1 \}$, so that $S=I \du J$.

For $i \in I$, define the sets
$$B_i= \{ y_{-\lceil \frac{i}{2} \rceil}, y_{\lfloor \frac{i}{2} \rfloor} \}
\qquad \mbox{ and } \qquad
E_i= \{ y_{\frac{t}{2}-\lceil \frac{i}{2} \rceil}, y_{\frac{t}{2}+\lfloor \frac{i}{2} \rfloor} \}.$$
For each $i \in J$, define the set
$$C_i= \{  y_{\frac{t}{2}-\lceil \frac{i+1}{2} \rceil}, y_{\frac{t}{2}+\lfloor \frac{i-1}{2} \rfloor} \}.$$
See Figure~\ref{fig:Se}. Observe that the sets $B_i$, for $i \in I$, are pairwise disjoint, as are the sets $E_i$, for $i \in I$, and the sets $C_i$, for $i \in J$. Moreover,
each of the sets $B_i$, $C_i$, and $E_i$ accommodates arcs of difference $\pm i$. Let $B=\bigcup_{i \in I} B_i$, $C=\bigcup_{j \in J} C_j$, and $E=\bigcup_{i \in I} E_i$, and observe that $B \cap C= \emptyset$, $B \cap E =\emptyset$,
$$\textstyle B \cup C \subseteq \{ y_i:  \lfloor \frac{d_B}{2} \rfloor \le i \le \frac{t}{2}- \lceil \frac{d_A+1}{2} \rceil \} \cup
\{ y_i:  \frac{t}{2}+\lfloor \frac{d_A-1}{2} \rfloor  \le i \le t- \lceil \frac{d_B}{2} \rceil \},$$
and
$$\textstyle B \cup E \subseteq \{ y_i:  \lfloor \frac{d_B}{2} \rfloor \le i \le \frac{t}{2}- \lceil \frac{d_B}{2} \rceil \} \cup
\{ y_i:  \frac{t}{2}+\lfloor \frac{d_B}{2} \rfloor  \le i \le t- \lceil \frac{d_B}{2} \rceil \}.$$

\smallskip

{\sc Subcase 2.1:} $2 \le a \le \frac{t-s+2}{2}$. Let $c=\frac{a}{2}$ and $A= \bigcup_{i=-c}^{c-1} A_i$, so $A$ is a
$(\vec{P}_1^{\langle a \rangle})$-subdigraph of $\dci(t; T)$ for $T=\{  \pm (\frac{t}{2}-1), \pm (\frac{t}{2}-3),\ldots, \pm (\frac{t}{2}-(2c-1)) \}$. Observe that $\frac{t}{2}-(2c-1)\ge d_A$ as $a \le \lfloor \frac{t-s+2}{2} \rfloor$.

Since $a \le \frac{s}{3}$, it is not difficult to show that $c < \min\{ \lceil \frac{d_A}{2} \rceil, \lceil \frac{d_B}{2} \rceil \}$. 
It follows that $V(A)$ is disjoint from $B \cup C$.

For $i \in I$, let $D_i$ be the directed 2-cycle with vertex set $B_i$, and for $j \in J-T$, let $D_j$ be the directed 2-cycle with vertex set $C_j$. Let $F=\bigcup_{i \in I} D_i \cup \bigcup_{j \in J-T} D_j $. Then $D=A \cup F$ is the desired $S$-orthogonal $(\vec{P}_1^{\langle a \rangle},\vec{C}_{2}^{\langle \frac{t-s-a}{2} \rangle})$-subdigraph of  $\dci(t;S)$.

\smallskip

{\sc Subcase 2.2:} $a=0$. With $T=\emptyset$, construct $F$ as in Subcase 2.1. Then $F$ is the desired $S$-orthogonal $(\vec{C}_{2}^{\langle \frac{t-s}{2} \rangle})$-subdigraph of  $\dci(t;S)$.

\smallskip

{\sc Subcase 2.3:} $a \ge \frac{t-s+4}{2}$. Let $c=\lfloor \frac{t-s+2}{4} \rfloor$, and
$A= \bigcup_{i=-c}^{c-1}$, so $A$ is a $T$-orthogonal
$(\vec{P}_1^{\langle 2c \rangle})$-subdigraph of $\dci(t; T)$ for $T=\{ \pm (\frac{t}{2}-1), \pm (\frac{t}{2}-3),\ldots, \pm d_A \}$.

As $\frac{t-s+4}{2} \le a \le \frac{s}{3}$, we have $t-s+4 \le \frac{2}{3}s<s$, and hence
$c < \lceil \frac{d_B}{2} \rceil$. It follows that
$V(A)$ is disjoint from $B \cup E$.

Let $b=a-2c$, and observe that $b$ is even, $b \ge 2$. Let $I' \subseteq I$ such that $|I'|=\frac{b}{2}$.
For $i \in I-I'$, let $D_i$ be the directed 2-cycle with vertex set $B_i$, and for $j \in I'$, let $P_j$ and $P_j'$ be the disjoint directed 1-paths with vertex sets in $B_j$ and $E_j$ that contain arcs of difference $j$ and $-j$, respectively. Let $F=\bigcup_{i \in I-I'} D_i \cup \bigcup_{j \in I'} (P_j \cup P_j')$. Then $D=A \cup F$ is the desired $S$-orthogonal $(\vec{P}_1^{\langle a \rangle},\vec{C}_{2}^{\langle \frac{t-s-a}{2} \rangle})$-subdigraph of  $\dci(t;S)$.

\bigskip

This completes all cases.
\end{proof}

\section{Main extension results}

We are now ready to sum up the proof of our main extension theorem, which we re-state below for convenience.

\begin{customthm}{\ref{the:main-ext}}
Let $\ell$ and $m_1,\ldots, m_{\ell}$ be positive integers such that  $m_i \ge 2$ for $i=1,\ldots, \ell$. Let $s=m_1+\ldots+m_{\ell}$ and let $t$ be an integer, $t>s$. Furthermore, let $a$ be the number of odd integers in the multiset $\{ \!\! \{ m_1,\ldots,m_{\ell} \} \!\! \}$, and assume that $a \le 2\lfloor \frac{t}{2} \rfloor -s$.

If OP$^\ast(m_1,\ldots,m_{\ell})$ has a solution, then the following also have a solution:
\begin{enumerate}[(1)]
\item OP$^\ast(m_1,\ldots,m_{\ell},t)$; and
\item if $t$ is even, OP$^\ast(m_1,\ldots,m_{\ell},2^{\langle \frac{t}{2} \rangle})$.
\end{enumerate}
\end{customthm}

\begin{proof} In both cases, for $i=1,\ldots,\ell$, let $s_i=\lfloor \frac{m_i}{2} \rfloor$, and $t_i=m_i-2s_i$. Observe that for each $i$, we have that $t_i=1$ if $m_i$ is odd, and $t_i=0$ if $m_i$ is even, so that $\sum_{i=1}^\ell t_i=a$.
Since $\sum_{i=1}^\ell m_i =s$, it follows that $a \equiv s \pmod{2}$, and as $m_i \ge 3$ if $t_i=1$, it follows that $a \le \lfloor \frac{s}{3} \rfloor$.

Let $r=s-\sum_{i=1}^\ell s_i$, and observe that  $r= s-\frac{1}{2}\sum_{i=1}^\ell (m_i-t_i)=s-\frac{s}{2}+\frac{a}{2} =\frac{s+a}{2}$. Hence $r > 0$, and by the assumption on $a$, it follows that $r \le \lfloor \frac{t}{2} \rfloor$.

\begin{enumerate}[(1)]
\item Assume first that $s \not\in \{ 3, 4 \}$ whenever $t$ is even.
    Let $m_{\ell+1}=t$, $s_{\ell +1}=r$, and $t_{\ell +1}=m_{\ell +1}-2s_{\ell +1}$. By the above observation, $0< s_{\ell +1} \le \lfloor \frac{t}{2} \rfloor$ and hence $t_{\ell +1} \ge 0$.

    Define $D \subseteq \ZZ_t^\ast$ as in Lemma~\ref{lem:long1}, and let $S=\ZZ_t^\ast-D$. Then $\dci(t;D)$ admits a $\vec{C}_t$-factorization by Lemma~\ref{lem:long1}. By Lemma~\ref{lem:long2},
    $\dci(t;S)$ admits an $S$-orthogonal $(\vec{P}_1^{\langle a \rangle},\vec{P}_{t-s-a})$-subdigraph, and hence also an $S$-orthogonal $(\vec{P}_0^{\langle \ell-a \rangle},\vec{P}_1^{\langle a \rangle},\vec{P}_{t-s-a})$-subdigraph.
    It follows that the conditions of Corollary~\ref{cor:main}, with $k=\ell+1$, are satisfied, and we conclude that OP$^\ast(m_1,\ldots,m_{\ell},t)$ has a solution.

    If $s=3$ and $t$ is even, then $(m_1,\ldots,m_{\ell})=(3)$, and OP$^\ast(3,t)$ has a solution by Lemma~\ref{lem:long-special}(a).

    If $s=4$ and $t$ is even, then we may assume that $(m_1,\ldots,m_{\ell})=(2,2)$ because OP$^\ast(4)$ does not have a solution. By Lemma~\ref{lem:long-special}(b),  OP$^\ast(2,2,t)$ has a solution whenever $t \ge 8$, while OP$^\ast(2,2,6)$ has a solution by Lemma~\ref{lem:special}.

\item Let $k=\ell+\frac{t}{2}$. As seen above, $0 < r \le \frac{t}{2}$. Let $s_i=1$ for $i=\ell+1,\ldots,\ell+r$, and $s_i=0$ for $i=\ell+r+1,\ldots,k$. For all $i=\ell+1,\ldots,k$, let $m_i=2$ and $t_i=m_i-2s_i$.

    Define $D \subseteq \ZZ_t^\ast$ as in Lemma~\ref{lem:2-cycles1}, and let $S=\ZZ_t^\ast-D$. Then $\dci(t;D)$ admits a $\vec{C}_2$-factorization by Lemma~\ref{lem:2-cycles1}.
    By Lemma~\ref{lem:2-cycles2},
    $\dci(t;S)$ admits an $S$-orthogonal $(\vec{P}_1^{\langle a \rangle},\vec{C}_{2}^{\langle \frac{t-s-a}{2} \rangle})$-subdigraph, and hence also an $S$-orthogonal $(\vec{P}_0^{\langle \ell-a+r \rangle},\vec{P}_1^{\langle a \rangle},\vec{C}_{2}^{\langle \frac{t-s-a}{2} \rangle})$-subdigraph.
    It follows that the conditions of Corollary~\ref{cor:main} are satisfied, and we conclude that OP$^\ast(m_1,\ldots,m_{\ell},2^{\langle \frac{t}{2} \rangle})$ has a solution.
\end{enumerate}
\end{proof}

In the rest of the section, we present some explicit existence results that are a consequence of Theorem~\ref{the:main-ext}.

\begin{cor}\label{cor:ext}
The following problems have a solution:
\begin{enumerate}[(a)]
\item OP$^\ast(m_1,m_2)$ if $2 \le m_1 < m_2$ and $m_1 \not\in \{ 4,6 \}$;
\item OP$^\ast(2^{\langle b \rangle},m)$ if $b \ge 1$.
\end{enumerate}
\end{cor}

\begin{proof}
\begin{enumerate}[(a)]
\item We use Theorem~\ref{the:main-ext}(1) with $\ell=1$, $s=m_1$, and $t=m_2$. Observe that the condition $a \le 2\lfloor \frac{t}{2} \rfloor -s$ is indeed satisfied, and that OP$^\ast(m_1)$ has a solution by Theorem~\ref{thm:Kn*}.
\item If $2b<m$,  we use Theorem~\ref{the:main-ext}(1) with $(m_1,\ldots,m_\ell)=(2^{\langle b \rangle})$ and $t=m$. Note that $a=0$ and that OP$^\ast(2^{\langle b \rangle})$ has a solution by Theorem~\ref{thm:Kn*}.

    If $m=2b$, the result follows from Corollary~\ref{cor:even}.

    Hence assume $m<2b$. If $m \not\in \{ 4,6 \}$, then OP$^\ast(m)$ has a solution by Theorem~\ref{thm:Kn*}, and
    we can use Theorem~\ref{the:main-ext}(2) with $\ell=1$, $s=m_1=m$, and $t=2b$. Note that $a=0$ if $s$ is even, and $a=1$ if $s$ is odd. In either case, $a \le 2 \lfloor \frac{t}{2} \rfloor -s$.   Hence OP$^\ast(m,2^{\langle b \rangle})$ has a solution.

    Let $m \in \{ 4,6 \}$, so that $b \ge \frac{m}{2}+1$. Now OP$^\ast(m,2^{\langle \frac{m}{2}+1 \rangle})$ has a
     solution by Lemma~\ref{lem:special}. Since OP$^\ast(2,m)$ and OP$^\ast(2^{\langle \frac{m}{2}+1 \rangle})$  have  solutions by (a) and Theorem~\ref{thm:Kn*}, respectively, OP$^\ast(m,2^{\langle \frac{m}{2}+2 \rangle})$
     has a solution by Theorem~\ref{thm:main-even}. Finally, since OP$^\ast(2,m)$ has a solution, OP$^\ast(m,2^{\langle b \rangle})$ for solution for $b\ge \frac{m}{2}+3$ by Theorem~\ref{the:main-ext}(2).
\end{enumerate}
\end{proof}

We are are now ready to summarize the proof of our almost-complete solution to the directed Oberwolfach problem with two tables.

\begin{customthm}{\ref{thm:main2}}
Let $m_1$ and $m_2$ be integers such that $2 \le m_1 \le m_2$. Then OP$^\ast(m_1,m_2)$ has a solution if and only if $(m_1,m_2) \ne (3,3)$, with a possible exception in the case that $m_1 \in \{ 4,6 \}$, $m_2$ is even, and $m_1+m_2 \ge 14$.
\end{customthm}

\begin{proof}
Assuming $m_1 \le m_2$, OP$^\ast(m_1,m_2)$  has a solution in the case that $m_1 \ge 3$, $m_1+m_2$ is odd, and $(m_1,m_2) \ne (4,5)$ by Corollary~\ref{cor:OP}; for  $(m_1,m_2) = (4,5)$ by Lemma~\ref{lem:special}; for $m_1=m_2 \ne 3$ by Theorem~\ref{thm:Kn*}; for  $2 \le m_1 < m_2$ and $m_1 \not\in \{ 4,6 \}$ by Corollary~\ref{cor:ext}(a); and for $(m_1,m_2) \in \{ (4,6),(4,8) \}$ by Lemma~\ref{lem:special}. Hence the result.
\end{proof}


In the remaining corollaries in this section, we use the following notation: for an integer $m$, we let $\delta(m)=0$ if $m$ is even, and $\delta(m)=1$ if $m$ is odd.

\begin{cor}\label{cor:ext2}
Let $2 \le m_1 \le m_2 \le m_3$ be integers such that OP$^\ast(m_1,m_2)$  has a solution, and assume $m_3 \ge m_1+m_2+\delta(m_1)+\delta(m_2)$.

Then OP$^\ast(m_1,m_2,m_3)$  has a solution.
\end{cor}

\begin{proof}
We attempt to use Theorem~\ref{the:main-ext} with $\ell=2$, $s=m_1+m_2$, and $t=m_3$, observing that  $a=\delta(m_1)+\delta(m_2)$. Since $m_3 \ge m_1+m_2+\delta(m_1)+\delta(m_2)$ by assumption, and $m_1+m_2+\delta(m_1)+\delta(m_2)$ is even, we have that
$$\textstyle 2 \lfloor \frac{t}{2} \rfloor = m_3-\delta(m_3) \ge m_1+m_2+\delta(m_1)+\delta(m_2)=s+a.$$
It follows that $a \le 2 \lfloor \frac{t}{2} \rfloor-s$, and if $t>s$, then OP$^\ast(m_1,m_2,m_3)$  has a solution by Theorem~\ref{the:main-ext}(1).

Hence we may assume that $t=s$. Then $m_1$, $m_2$, and $m_3$ are all even, and $m_3=m_1+m_2$. If $m_3 \not\in \{ 4,6 \}$, then OP$^\ast(m_3)$  has a solution, and hence OP$^\ast(m_1,m_2,m_3)$  has a solution by Theorem~\ref{thm:main-even}. If $m_3=4$, then $(m_1,m_2)=(2,2)$, and if $m_3=6$, then $(m_1,m_2)=(2,4)$. Since OP$^\ast(2,2,4)$ and OP$^\ast(2,4,6)$ have solutions by Lemmas~\ref{lem:(2,2,4)} and \ref{lem:special}, respectively, the proof is complete.
\end{proof}

We remark that for $n=m_1+m_2+m_3 \le 100$ with $n$ odd, more existence results on OP$^\ast(m_1,m_2,m_3)$ follow from Corollary~\ref{cor:OP}.

\begin{cor}\label{cor:ext3}
Let $3 \le m_1 \le m_2$ and $b \ge 1$ be integers.
Then OP$^\ast(m_1,m_2,2^{\langle b \rangle})$  has a solution in each of the following cases:
\begin{enumerate}[(i)]
\item OP$^\ast(m_1,m_2)$ has a solution, and $2b \ge m_1+m_2+\delta(m_1)+\delta(m_2)$ or $2b \le m_2-m_1-\delta(m_1)$;
\item $(m_1,m_2) =(3,3)$;
\item $m_1 \in \{ 4,6 \}$, $m_2$ is even, $m_1+m_2 \ge 14$, and $2b \ge m_1+m_2+4$ or $2b \le m_2-m_1$.
\end{enumerate}
\end{cor}

\begin{proof}
\begin{enumerate}[(i)]
\item If $2b \ge m_1+m_2+\delta(m_1)+\delta(m_2)$, then we attempt to use Theorem~\ref{the:main-ext}(2) with $\ell=2$, $s=m_1+m_2$, and $t=2b$. Note that $a=\delta(m_1)+\delta(m_2) \le 2b-s$. If $t>s$, then OP$^\ast(m_1,m_2,2^{\langle b \rangle})$  has a solution by Theorem~\ref{the:main-ext}(2). If, however, $t=s$, then $m_1$ and $m_2$ are both even and $2b=m_1+m_2$, and since OP$^\ast(2^{\langle b \rangle})$ and OP$^\ast(m_1,m_2)$ both have solutions, the result follows from Theorem~\ref{thm:main-even}.

    If $2b \le m_2-m_1-\delta(m_1)$, then we attempt to use Theorem~\ref{the:main-ext}(1) with $\ell=b+1$, $s=m_1+2b$, and $t=m_2$. Note that $a=\delta(m_1) \le t-s$. If $t>s$, then OP$^\ast(m_1,2^{\langle b \rangle},m_2)$  has a solution by Theorem~\ref{the:main-ext}(1). If, however, $t=s$, then $m_1$ is even and $m_2=m_1+2b$. If $m_2 \not \in \{ 4,6 \}$, then OP$^\ast(m_2)$ has a solution by Theorem~\ref{thm:Kn*}, as does OP$^\ast(m_1,2^{\langle b \rangle})$ by Corollary~\ref{cor:ext}(b), and the result follows from Theorem~\ref{thm:main-even}. Since $m_1 \ge 3$, we know $m_2 \ne 4$. Finally, if $m_2=6$, then $m_1=4$ and $b=1$, and OP$^\ast(2,4,6)$ has a solution by Lemma~\ref{lem:special}.

\item OP$^\ast(2,3,3)$ has a solution by Lemma~\ref{lem:special}. Using Theorem~\ref{the:main-ext}(2) with $\ell=3$, $s=8$, and $t=2(b-1)$, assuming $b \ge 6$, we have that $2=a \le t-s$, and OP$^\ast(3,3,2^{\langle b \rangle})$ has a solution. For $2 \le b \le 5$, the result follows from  Lemma~\ref{lem:special}.


\item By Corollary~\ref{cor:ext}(2), OP$^\ast(2,m_1)$ has a solution, and since $m_2 \ge 2+m_1$, we know that OP$^\ast(2,m_1,m_2)$ has a solution by Corollary~\ref{cor:ext2}.

    If $2b \ge m_1+m_2+4$, then we attempt to use Theorem~\ref{the:main-ext}(2) with $\ell=3$, $s=2+m_1+m_2$, and $t=2(b-1)$. Note that $0=a \le t-s$. If $t>s$, then OP$^\ast(m_1,m_2,2^{\langle b \rangle})$  has a solution by Theorem~\ref{the:main-ext}(2), and if $t=s$, then the result follows from Theorem~\ref{thm:main-even}.

   If $2b \le m_2-m_1$, then we attempt to use Theorem~\ref{the:main-ext}(1) with $\ell=b+1$, $s=m_1+2b$, and $t=m_2$. Note that $0=a\le t-s$. If $t>s$, then OP$^\ast(m_1,2^{\langle b \rangle},m_2)$  has a solution by Theorem~\ref{the:main-ext}(1), and if $t=s$, then the result follows from Theorem~\ref{thm:main-even} because OP$^\ast(m_1,2^{\langle b \rangle})$ and OP$^\ast(m_2)$ have solutions by Corollary~\ref{cor:ext}(2) and Theorem~\ref{thm:Kn*}, respectively.
\end{enumerate}
\end{proof}


Repeatedly applying Theorem~\ref{the:main-ext}(1) we obtain the following.

\begin{cor}\label{cor:global}
For $i=1,2,\ldots,k$, let $m_i \ge 2$. Assume $m_1 \not \in \{ 4,6 \}$ and that for each $\ell<k$, we have $m_{\ell+1} -\delta(m_{\ell+1}) > \sum_{i=1}^{\ell} (m_i + \delta(m_i))$. Then
OP$^\ast(m_1,\ldots,m_k)$ has a solution.
\end{cor}

\begin{proof}
Since $m_1 \not \in \{ 4,6 \}$, by Theorem~\ref{thm:Kn*}, OP$^\ast(m_1)$ has a solution.

Fix any $\ell <k$, and assume OP$^\ast(m_1,\ldots,m_{\ell})$ has a solution. Let $a$ be the number of odd integers in the multiset $\{ \!\! \{ m_1,\ldots,m_{\ell} \} \!\! \}$, let $s=\sum_{i=1}^{\ell} m_i$ and $t=m_{\ell+1}$. Then $a=\sum_{i=1}^{\ell} \delta(m_i) < m_{\ell+1} -\delta(m_{\ell+1}) - \sum_{i=1}^{\ell} m_i=2 \lfloor \frac{t}{2} \rfloor -s$, so the conditions of Theorem~\ref{the:main-ext}(1) are satisfied. It follows that OP$^\ast(m_1,\ldots,m_{\ell},m_{\ell+1})$ has a solution.

The result follows by induction.
\end{proof}

\section{Solutions to small cases of OP$^\ast$}\label{sec:small}

In this section, we solve the directed Oberwolfach problem for orders up to 13.

\begin{theo}\label{thm:small}
Let $2 \le m_1 \le \ldots \le m_k$ be integers, and assume that $n=m_1+ \ldots + m_k \le 13$. Then OP$^\ast(m_1,\ldots,m_k)$ has a solution if and only if $(m_1,\ldots,m_k) \not\in \{ (4),(6),(3,3) \}$.
\end{theo}

\begin{proof} For the uniform case, that is, when $m_1=\ldots = m_k$, the result follows from Theorem~\ref{thm:Kn*}, while for $n$ odd, $m_1\ge3$, and $(m_1,\ldots,m_k) \not\in \{ (4,5),(3,3,5) \}$, the result follows from Corollary~\ref{cor:OP}. The remaining cases are listed in the tables below.

\bigskip

\begin{tabular}{|c|c|l|}
\hline
$n=m_1+ \ldots + m_k$ & $(m_1,\ldots,m_k)$ & OP$^\ast(m_1,\ldots,m_k)$ has a solution by... \\ \hline \hline
5 &  $(2,3)$ &  Theorem~\ref{thm:main2} \\ \hline
6 & $(2,4)$ &  Theorem~\ref{thm:main2}  \\ \hline
7 & $(2,5)$ &  Theorem~\ref{thm:main2}  \\
& $(2,2,3)$ &  Corollary~\ref{cor:ext}(b)  \\ \hline
\end{tabular}

\begin{tabular}{|c|c|l|}
\hline
$n=m_1+ \ldots + m_k$ & $(m_1,\ldots,m_k)$ & OP$^\ast(m_1,\ldots,m_k)$ has a solution by... \\ \hline \hline
8 & $(2,6)$ &  Theorem~\ref{thm:main2}  \\
& $(3,5)$ &  Theorem~\ref{thm:main2}  \\
 & $(2,2,4)$ &  Lemma~\ref{lem:(2,2,4)} \\
 & $(2,3,3)$ &  Lemma~\ref{lem:special} \\ \hline
9 &  $(2,7)$ &  Theorem~\ref{thm:main2}  \\
&  $(4,5)$ &   Theorem~\ref{thm:main2} \\
 & $(2,2,5)$ &  Corollary~\ref{cor:ext}(b) \\
 & $(2,3,4)$ &  Lemma~\ref{lem:1-rotational-q} and Appendix~\ref{app1} \\
  & $(2,2,2,3)$ &  Corollary~\ref{cor:ext}(b) \\ \hline
 10 & $(2,8)$ &  Theorem~\ref{thm:main2}  \\
& $(3,7)$ &  Theorem~\ref{thm:main2}  \\
&  $(4,6)$ &  Theorem~\ref{thm:main2}  \\
 & $(2,2,6)$ &  Corollary~\ref{cor:ext}(b) \\
 & $(2,3,5)$ &  Lemma~\ref{lem:1-rotational-q} and Appendix~\ref{app1} \\
  & $(2,4,4)$ &  Lemma~\ref{lem:1-rotational-q} and Appendix~\ref{app1} \\
 & $(3,3,4)$ &   Lemma~\ref{lem:1-rotational-q} and Appendix~\ref{app3} \\
 & $(2,2,2,4)$ &  Corollary~\ref{cor:ext}(b)  \\
  & $(2,2,3,3)$ &  Lemma~\ref{lem:special}  \\ \hline
11 & $(2,9)$ &  Theorem~\ref{thm:main2}  \\
 & $(2,2,7)$ &  Corollary~\ref{cor:ext}(b) \\
 & $(2,3,6)$ &  Theorems~\ref{thm:main2} and \ref{the:main-ext}(1) \\
  & $(2,4,5)$ &  Lemma~\ref{lem:1-rotational-q} and Appendix~\ref{app1} \\
  & $(3,3,5)$ &  Lemma~\ref{lem:1-rotational-q} and Appendix~\ref{app1} \\
 & $(2,2,2,5)$ &  Corollary~\ref{cor:ext}(b)  \\
  & $(2,2,3,4)$ &  Lemma~\ref{lem:1-rotational-q} and Appendix~\ref{app1} \\
  & $(2,3,3,3)$ &  Lemma~\ref{lem:1-rotational-q} and Appendix~\ref{app1} \\
  & $(2,2,2,2,3)$ &  Corollary~\ref{cor:ext}(b) \\
\hline
12 & (2,10) &  Theorem~\ref{thm:main2} \\
& (3,9) &  Theorem~\ref{thm:main2} \\
& (4,8) &  Lemma~\ref{lem:special} \\
& (5,7) &  Theorem~\ref{thm:main2}\\
& (2,2,8) &  Corollary~\ref{cor:ext}(b) \\
& (2,3,7) &  Theorems~\ref{thm:main2} and \ref{the:main-ext}(1) \\
& (2,4,6) &  Lemma~\ref{lem:special} \\
& (2,5,5) &   Lemma~\ref{lem:1-rotational-q} and Appendix~\ref{app1} \\
& (3,3,6) &  Appendix~\ref{appq} \\
& (3,4,5) &  Appendix~\ref{appq} \\
& (2,2,2,6) &  Corollary~\ref{cor:ext}(b) \\
& (2,2,3,5) &  Lemma~\ref{lem:1-rotational-q} and Appendix~\ref{app1} \\
& (2,2,4,4) &  Theorems~\ref{thm:main2} and \ref{thm:main-even} \\
& (2,3,3,4) &  Lemma~\ref{lem:1-rotational-q} and Appendix~\ref{app1} \\
& (2,2,2,2,4) &  Corollary~\ref{cor:ext}(b) \\
& (2,2,2,3,3) &  Lemma~\ref{lem:special} \\\hline
\end{tabular}

\begin{tabular}{|c|c|l|}
\hline
$n=m_1+ \ldots + m_k$ & $(m_1,\ldots,m_k)$ & OP$^\ast(m_1,\ldots,m_k)$ has a solution by... \\ \hline \hline
13 & (2,11) &  Theorem~\ref{thm:main2} \\
& (2,2,9) &  Corollary~\ref{cor:ext}(b) \\
& (2,3,8) &  Theorems~\ref{thm:main2} and \ref{the:main-ext}(1)\\
& (2,4,7) &  Theorems~\ref{thm:main2} and \ref{the:main-ext}(1) \\
& (2,5,6)  &  Lemma~\ref{lem:1-rotational-q} and Appendix~\ref{app1} \\
& (2,2,2,7) &  Corollary~\ref{cor:ext}(b) \\
& (2,2,3,6) &  Lemma~\ref{lem:1-rotational-q} and Appendix~\ref{app1} \\
& (2,2,4,5) &  Lemma~\ref{lem:1-rotational-q} and Appendix~\ref{app1} \\
& (2,3,3,5) &  Lemma~\ref{lem:1-rotational-q} and Appendix~\ref{app1} \\
& (2,3,4,4) &  Lemma~\ref{lem:1-rotational-q} and Appendix~\ref{app1}\\
& (2,2,2,2,5) &  Corollary~\ref{cor:ext}(b) \\
& (2,2,2,3,4) &  Lemma~\ref{lem:1-rotational-q} and Appendix~\ref{app1} \\
& (2,2,3,3,3) &  Lemma~\ref{lem:1-rotational-q} and Appendix~\ref{app1} \\
& (2,2,2,2,2,3) &  Corollary~\ref{cor:ext}(b) \\  \hline
\end{tabular}

\end{proof}

For $14 \le n \le 16$, we have obtained similar, though incomplete results. Namely, we showed OP$^\ast(m_1,\ldots,m_k)$ has a solution, except possibly for $(m_1,\ldots,m_k) \in \{ (4,10), (6,8),$ $(3,3,8), (3,4,7),(3,5,6),(4,4,6),(4,5,5),(3,3,3,5),(3,3,4,4)\}$.

The difficulty with these outstanding cases (all with $n=14$) is as follows. First, these cases do not fall into any of the families solved by the results of this paper, so a direct construction is needed. A simple 1-rotational approach with $q=1$ described in Lemma~\ref{lem:1-rotational-q} cannot possibly work when $n$ is even and all $m_i \ge 3$. Namely, if in this case, a $(\vec{C}_{m_1},\ldots,\vec{C}_{m_k})$-factor $F$ satisfying the conditions of Lemma~\ref{lem:1-rotational-q} with $q=1$ existed, then $\dci(n-1;\ZZ_{n-1}^\ast)$ would admit a $\ZZ_{n-1}^\ast$-orthogonal subdigraph consisting of disjoint directed cycles and a directed path. Since the differences of the arcs in each directed cycle add up to 0 modulo $n-1$, and the differences in $\ZZ_{n-1}^\ast$ add up to 0 as well, the differences of the arcs in the remaining directed path would have to add up to 0 as well, which is a contradiction.

If $n-1$ has a small divisor $q >1$, then the next simplest approach would be a base-$q$ 1-rotational construction from Lemma~\ref{lem:1-rotational-q}. Unfortunately, for $n=14$, we have that $n-1$ is a prime, so the smallest possible $q=13$. This essentially means that we are looking for a directed 2-factorization without symmetry, which is virtually impossible by hand, and computationally hard.

Finally, applying Theorem~\ref{the:main-ext} to Theorem~\ref{thm:small}, we immediately obtain the following.

\begin{cor}\label{cor:small}
Let $2 \le m_1 \le \ldots \le m_{\ell}<t$ be integers such that $s=m_1+ \ldots + m_{\ell} \le 13$. Furthermore, assume that the following all hold:
\begin{enumerate}[(i)]
\item $(m_1,\ldots,m_{\ell}) \not\in \{ (4),(6),(3,3) \}$;
\item $t>s$; and
\item $\delta(m_1)+\ldots+\delta(m_{\ell}) \le t-\delta(t)-s$.
\end{enumerate}

Then OP$^\ast(m_1,\ldots,m_{\ell},t)$, as well as OP$^\ast(m_1,\ldots,m_{\ell},2^{\langle \frac{t}{2} \rangle})$ if $t$ is even, has a solution.
\end{cor}

\section{Conclusion}

In this paper, we introduced a recursive method that allows us to construct solutions to larger cases of OP$^\ast$ from solutions to smaller cases. In particular, we showed how to extend the directed 2-factor in a known solution by a long cycle, as well as by a relatively large number of 2-cycles. Among many other explicit infinite families of cases that we were able to successfully address,  this method yields an almost-complete solution to the two-table directed Oberwolfach problem. We were also able to settle all cases with up to 13 vertices.

In closing, we would like to propose several future research directions. The most compelling would be to complete the solution to the two-table OP$^\ast$, which would be analogous to Theorem~\ref{thm:Tra} \cite{BryDan,Gvo,Hag,Tra}. Since the outstanding cases are all bipartite, such a solution could perhaps be obtained as a special case of the solution to the general bipartite case, analogous to the results of \cite{BryDan,Hag} for OP.

Another natural direction would be to push the upper bound on $n$ in Theorem~\ref{thm:small} in order to obtain a result more comparable to that of Theorem~\ref{thm:OPsmall}. While larger cases in Lemma~\ref{lem:special} and Appendix~\ref{app} were obtained computationally, more sophisticated algorithms and more computing power will be needed to solve cases with $n \ge 14$.

It is well known and easy to see that a solution to OP$^\ast(m_1,\ldots,m_t)$ gives rise to a solution to OP$\!_2(m_1,\ldots,m_t)$, which asks for a $(C_{m_1},\ldots,C_{m_t})$-factorization of $2K_n$, the two-fold complete graph with $n=m_1+\ldots+m_t$ vertices. Thus both the explicit solutions to OP$^\ast$, as well as the recursive constructions presented in this paper, extend to OP$\!_2$. It would be of interest to collect the known results on OP$\!_2$ (of which there are many more than for OP$^\ast$), and use these solutions as base cases for our recursive method to see what explicit new solutions can be obtained.

While it is generally believed that the directed version of OP is harder than the undirected version, the discovery of our new recursive method, which in its present form does not apply to undirected simple graphs, may cast a shadow of doubt on this belief. On the other hand, we propose that it might be possible to develop a more limited version of our recursive construction that would, in fact, apply to simple graphs and therefore lead to new solutions for  OP.


\subsection*{Acknowledgements}

The authors wish to thank the anonymous referees for their thoughtful comments, and Kelsey Gasior for mentoring and financially supporting the first author. The first author would like to thank Allah for the knowledge, ideas, and opportunities He has given her. She would also like to thank the second author for her continuous kindness and guidance. The second author gratefully acknowledges support by the Natural Sciences and Engineering Research Council of Canada (NSERC), Discovery Grant RGPIN-2022-02994.


\appendix

\section{Solutions to special small cases}\label{app}

\subsection{Solutions with a single starter}\label{app1}

In each of the following problems, we give a  starter 2-factor in a base-1 1-rotational solution, as described in Lemma~\ref{lem:1-rotational-q}.

\begin{itemize}
\item OP$^\ast(2,3,4)$:
    $ u_1 \, u_7 \, u_1 \; \cup \; u_0 \, u_4 \, u_{\infty} \, u_0 \; \cup \; u_2 \, u_3 \, u_6 \, u_5 \, u_2$;

\item OP$^\ast(2,3,5)$:
    $ u_0 \, u_{\infty} \, u_0  \; \cup \; u_2 \, u_3 \, u_5 \, u_2
    \; \cup \; u_1 \, u_6 \, u_4 \; u_8 \, u_7 \, u_1$;

\item OP$^\ast(2,4,4)$:
    $ u_0 \, u_{\infty} \, u_0  \; \cup \;  u_1 \, u_7 \, u_5 \, u_6 \, u_1 \; \cup \;  u_2 \, u_4 \, u_3 \, u_8 \, u_2$;

\item OP$^\ast(2,4,5)$:
    $ u_1 \, u_2 \, u_1  \; \cup \;  u_3 \, u_{\infty} \, u_8 \, u_6 \, u_3 \; \cup \;  u_0 \, u_4 \, u_7 \, u_9 \, u_5 \, u_0$;

\item OP$^\ast(3,3,5)$:
    $ u_0 \, u_{\infty} \, u_5 \, u_0  \; \cup \; u_3 \, u_4 \, u_6 \, u_3
    \; \cup \; u_1 \, u_9 \, u_2 \; u_8 \, u_7 \, u_1$;

\item OP$^\ast(2,2,3,4)$:
    $  u_1 \, u_9 \, u_1 \; \cup \; u_2 \, u_8 \, u_2  \; \cup \; u_0 \, u_{\infty} \, u_5 \, u_0   \; \cup \; u_3 \, u_4 \, u_7 \, u_6 \, u_3$;

\item OP$^\ast(2,3,3,3)$:
    $  u_3 \, u_7 \, u_3 \; \cup \;  u_0 \, u_{\infty} \, u_5 \, u_0   \; \cup \; u_1 \, u_2 \, u_4 \, u_1
    \; \cup \; u_6 \, u_9 \, u_8 \, u_6$;

\item OP$^\ast(2,5,5)$:
    $  u_0 \, u_{\infty} \, u_0 \; \cup \;  u_1 \, u_8 \, u_2 \, u_4 \, u_3 \, u_1   \; \cup \; u_5 \, u_9 \, u_6 \, u_7
    \; u_{10} \, u_5 $;

\item OP$^\ast(2,2,3,5)$:
    $  u_0 \, u_{\infty} \, u_0 \; \cup \;  u_5 \, u_6 \, u_5 \; \cup \; u_2 \, u_7 \, u_4 \, u_2   \; \cup \; u_1 \, u_8 \, u_{10} \, u_3 \; u_9 \, u_1 $;

\item OP$^\ast(2,3,3,4)$:
    $  u_0 \, u_{\infty} \, u_0 \; \cup \;  u_1 \, u_5 \, u_2 \, u_1 \; \cup \; u_4 \, u_{10} \, u_8 \, u_4   \; \cup \; u_3 \, u_6 \, u_7 \; u_9 \, u_3 $;

\item OP$^\ast(2,5,6)$:
$ u_3 \, u_6 \, u_3 \; \cup \; u_0 \, u_7 \, u_5 \, u_{11} \, u_1 \, u_0 \; \cup \;
u_2 \, u_{10} \, u_{\infty} \, u_4 \, u_8 \, u_9 \, u_2  $;

\item OP$^\ast(2,2,3,6)$:
$ u_5 \, u_7 \, u_5 \; \cup \; u_{10} \, u_{11} \, u_{10} \; \cup \;  u_1 \, u_4 \, u_8 \, u_1 \; \cup \;
u_0 \, u_{\infty} \, u_6 \, u_2 \, u_9 \, u_3 \, u_0$;

\item OP$^\ast(2,2,4,5)$:
$ u_0 \, u_{11} \, u_0 \; \cup \; u_1 \, u_5 \, u_1 \; \cup \; u_2 \, u_7 \, u_{10} \, u_4 \, u_2 \; \cup \;
u_3 \, u_{\infty} \, u_9 \, u_6 \, u_8 \, u_3$;

\item OP$^\ast(2,3,3,5)$:
$ u_2 \, u_4 \, u_2 \; \cup \; u_0 \, u_ 1 \, u_7 \, u_0 \; \cup \; u_3 \, u_{10} \, u_6 \, u_3 \; \cup \;
u_5 \, u_9 \, u_8 \, u_{11} \, u_{\infty} \, u_5$;

\item OP$^\ast(2,3,4,4)$:
$ u_6 \, u_9 \, u_6 \; \cup \; u_1 \, u_8 \, u_2 \, u_1 \; \cup \; u_0 \, u_{10} \, u_3 \, u_4 \, u_ 0 \; \cup \;
u_5 \, u_7 \, u_{11} \, u_{\infty} \, u_5$;

\item OP$^\ast(2,2,2,3,4)$:
$ u_1 \, u_{10} \, u_1 \; \cup \; u_4 \, u_{11} \, u_ 4 \; \cup \; u_5 \, u_6 \, u_5 \; \cup \; u_3 \, u_7 \, u_9 \, u_3 \; \cup \; u_0 \, u_8 \, u_{\infty} \, u_2 \, u_0$;

\item OP$^\ast(2,2,3,3,3)$:
$ u_0 \, u_5 \, u_0 \; \cup \; u_4 \, u_6 \, u_4 \; \cup \; u_1 \, u_{10} \, u_ 2 \, u_1 \; \cup \; u_3 \, u_9 \, u_{\infty} \, u_3\; \cup \; u_7 \, u_8 \, u_{11} \, u_7$;
\end{itemize}

\subsection{Solutions with three starters}\label{app3}

In the following problem, we give three starter 2-factors in a base-3 1-rotational solution, as described in Lemma~\ref{lem:1-rotational-q}.

\begin{itemize}
\item OP$^\ast(3,3,4)$:
    \begin{eqnarray*}
    F_0 &=& u_0 \, u_5 \, u_2 \, u_0 \; \cup \; u_1 \, u_{\infty} \, u_3 \, u_1  \; \cup \;  u_4 \,  u_6 \, u_8 \, u_7 \, u_4, \\
    F_1 &=& u_0 \, u_3 \, u_4 \, u_0 \; \cup \; u_5 \, u_6 \, u_{\infty} \, u_5  \; \cup \;  u_1 \,  u_8 \, u_2 \, u_7 \, u_1, \\
    F_2 &=& u_2 \, u_4 \, u_3 \, u_2 \; \cup \; u_7 \, u_8 \, u_{\infty} \, u_7  \; \cup \;  u_0 \,  u_6 \, u_1 \, u_5 \, u_0.
    \end{eqnarray*}

\end{itemize}

\subsection{Solutions without symmetry}\label{appq}

In each of the following problems, we list the directed 2-factors that form a solution.

\begin{itemize}
\item OP$^\ast(3,3,6)$:
    \begin{eqnarray*}
    F_0 &=& u_3 \, u_{\infty} \, u_{10} \, u_3 \; \cup \; u_4 \, u_6 \, u_7 \, u_4  \; \cup \;  u_0 \,  u_8 \, u_9 \, u_2 \, u_5 \, u_1 \, u_0, \\
    F_1 &=& u_1 \, u_2 \, u_{10} \, u_1 \; \cup \; u_3 \, u_6 \, u_8 \, u_3  \; \cup \;  u_0 \,  u_4 \, u_7 \, u_9 \, u_{\infty} \, u_5 \, u_0, \\
    F_2 &=& u_1  \, u_{10} \, u_4 \, u_1 \; \cup \; u_7 \, u_{\infty} \, u_8 \, u_7  \; \cup \;  u_0 \,  u_6 \, u_2 \, u_9 \, u_5 \, u_3 \, u_0, \\
    F_3 &=& u_0  \, u_{10} \, u_9 \, u_0 \; \cup \; u_7 \, u_8 \, u_{\infty} \, u_7  \; \cup \;  u_1 \,  u_5 \, u_4 \, u_2 \, u_6 \, u_3 \, u_1, \\
    F_4 &=& u_1  \, u_6 \, u_{\infty} \, u_1 \; \cup \; u_5 \, u_9 \, u_7 \, u_5  \; \cup \;  u_0 \,  u_3 \, u_2 \, u_8 \, u_4 \, u_{10} \, u_0, \\
    F_5 &=& u_4  \, u_5 \, u_{\infty} \, u_4 \; \cup \; u_3 \, u_{10} \, u_7 \, u_3  \; \cup \;  u_0 \,  u_2 \, u_1 \, u_9 \, u_8 \, u_6 \, u_0, \\
    F_6 &=& u_2  \, u_7 \, u_{10} \, u_2 \; \cup \; u_4 \, u_{9} \, u_6 \, u_4  \; \cup \;  u_0 \,  u_1 \, u_{\infty} \, u_3 \, u_5 \, u_8 \, u_0, \\
    F_7 &=& u_2  \, u_3 \, u_7 \, u_2 \; \cup \; u_5 \, u_{10} \, u_8 \, u_5  \; \cup \;  u_0 \,  u_{\infty} \, u_6 \, u_9  \, u_1 \, u_4 \, u_0, \\
    F_8 &=& u_3  \, u_9 \, u_4 \, u_3 \; \cup \; u_5 \, u_6 \, u_{10} \, u_5  \; \cup \;  u_0  \, u_7 \, u_1  \, u_8 \, u_2 \,  u_{\infty} \, u_0, \\
    F_9 &=& u_1  \, u_3 \, u_8 \, u_1 \; \cup \; u_2 \, u_4 \, u_{\infty} \, u_2  \; \cup \;  u_0  \, u_9 \, u_{10}  \, u_6 \, u_5 \, u_7 \, u_0, \\
    F_{10} &=& u_0 \, u_5  \, u_2 \, u_0 \; \cup \; u_1 \, u_7 \, u_6 \, u_1  \; \cup \;  u_3  \, u_4 \, u_8 \, u_{10}  \, u_{\infty} \, u_9 \, u_3;
    \end{eqnarray*}

\item OP$^\ast(3, 4, 5)$:
\begin{eqnarray*}
F_0 &=& u_9 \, u_{10} \, u_{\infty} \, u_9 \; \cup \;  u_3 \, u_7 \, u_4 \, u_5 \,  u_3 \; \cup \;  u_0 \, u_6 \, u_8 \, u_1 \, u_2 \, u_0, \\
F_1 &=& u_7 \, u_9 \, u_{\infty} \, u_7 \; \cup \;  u_0 \, u_{10} \, u_3 \, u_8 \, u_0 \; \cup \;   u_1 \, u_4 \, u_6 \, u_5 \, u_2 \, u_1, \\
F_2 &=& u_2 \, u_7 \, u_{10} \, u_2 \; \cup \;  u_1 \, u_{\infty} \, u_4 \, u_9 \, u_1 \; \cup \;  u_0 \, u_8 \, u_6 \, u_3 \, u_5 \, u_0, \\
F_3 &=& u_1 \, u_9 \, u_4 \, u_1 \; \cup \;  u_2 \, u_{10} \, u_8 \, u_7 \, u_2 \; \cup \;  u_0 \, u_5 \, u_6 \, u_{\infty} \, u_3 \, u_0, \\
F_4 &=& u_0 \, u_2 \, u_{\infty} \, u_0 \; \cup \;  u_5 \, u_9 \, u_8 \, u_{10} \, u_5 \; \cup \;  u_1 \, u_6 \, u_4 \, u_7 \, u_3 \, u_1, \\
F_5 &=& u_5 \, u_{10} \, u_9 \, u_5 \; \cup \;  u_2 \, u_3 \, u_4 \, u_8 \, u_2 \; \cup \;  u_0 \, u_{\infty} \, u_6 \, u_1 \, u_7 \, u_0, \\
F_6 &=& u_2 \, u_8 \, u_4 \, u_2 \; \cup \;   u_0 \, u_3 \, u_9 \, u_6 \, u_0 \; \cup \;  u_1 \, u_5 \, u_7 \, u_{\infty} \, u_{10} \, u_1, \\
F_7 &=& u_2 \, u_9 \, u_3 \, u_2 \; \cup \;   u_1 \, u_8 \, u_5 \, u_{\infty} \, u_1 \; \cup \;  u_0 \, u_7 \, u_6 \, u_{10} \, u_4 \, u_0, \\
F_8 &=& u_3 \, u_{\infty} \, u_8 \, u_3 \; \cup \;  u_2 \, u_4 \, u_{10} \, u_6 \, u_2 \; \cup \;  u_0 \, u_9 \, u_7 \, u_5 \, u_1 \, u_0, \\
F_9 &=& u_1 \, u_{10} \, u_7 \, u_1 \; \cup \;  u_2 \, u_5 \, u_8 \, u_{\infty} \,  u_2 \; \cup \;  u_0 \, u_4 \, u_3 \, u_6 \, u_9 \,  u_0, \\
F_{10} &=& u_4 \, u_{\infty} \, u_5 \, u_4 \; \cup \;  u_0 \, u_1 \, u_3 \, u_{10} \, u_0 \; \cup \;   u_2 \, u_6 \, u_7 \, u_8 \, u_9 \, u_2.
\end{eqnarray*}

\end{itemize}


\small

\begin{thebibliography}{99}

\bibitem{AdaBry1} P. Adams, D. Bryant, Resolvable directed cycle systems of all indices for cycle length 3 and 4, unpublished.

\bibitem{AdaBry} P.  Adams, D. Bryant,
Two-factorisations of complete graphs of orders fifteen and seventeen,
{\em Australas. J. Combin.} {\bf 35} (2006), 113--118.

\bibitem{AlsGavSaj} B. Alspach, H. Gavlas, M. \v{S}ajna, H. Verrall,  Cycle decompositions IV: Complete directed graphs and fixed length directed cycles, {\em J. Combin. Theory Ser. A} {\bf 103} (2003),  165--208.

\bibitem{AlsHag} B. Alspach, R. H\"{a}ggkvist, Some observations on the Oberwolfach problem,
{\em J. Graph Theory} {\bf 9} (1985), 177--187.

\bibitem{AlsSch} B. Alspach, P. J. Schellenberg, D. R. Stinson, D. Wagner, The Oberwolfach problem and factors of uniform odd length cycles, {\em J. Combin. Theory Ser. A} {\bf 52} (1989), 20--43.

\bibitem{BenZha} F. E. Bennett, X. Zhang,  Resolvable Mendelsohn designs with block size $4$, {\em Aequationes Math.} {\bf  40} (1990), 248--260.



\bibitem{BerFavMah} J.-C. Bermond, O. Favaron, M. Mah\'{e}o, Hamiltonian decomposition of Cayley graphs of degree 4, {\em J. Combin. Theory Ser. B} {\bf 46} (1989), 142--153.

\bibitem{BerGerSot} J.-C. Bermond, A. Germa, D. Sotteau, Resolvable decomposition of $K_n^\ast$,
{\em J. Combin. Theory Ser. A} {\bf 26} (1979), 179--185.

\bibitem{BryDan} D. Bryant, P. Danziger,
On bipartite 2-factorizations of $K_n-I$ and the Oberwolfach problem,
{\em J. Graph Theory} {\bf  68} (2011),  22--37.

\bibitem{BurFraSaj} A. Burgess, N. Franceti\'{c}, M. \v{S}ajna,  On the directed Oberwolfach Problem with equal cycle lengths: the odd case, {\em Australas. J. Combin.} {\bf 71} (2018), 272--292.

\bibitem{BurSaj} A. Burgess, M. \v{S}ajna,
On the directed Oberwolfach problem with equal cycle lengths,
{\em Electron. J. Combin.} {\bf 21} (2014),  Paper 1.15, 14 pp.


\bibitem{Dez} A. Deza, F. Franek, W. Hua, M. Meszka, A. Rosa,
Solutions to the Oberwolfach problem for orders 18 to 40,
{\em J. Combin. Math. Combin. Comput.} {\bf  74} (2010), 95--102.

\bibitem{FraHol} F. Franek, J. Holub, A. Rosa,
Two-factorizations of small complete graphs II: The case of 13 vertices,
 {\em J. Combin. Math. Combin. Comput.} {\bf 51} (2004), 89--94.

\bibitem{FraRos} F. Franek, A. Rosa,
Two-factorizations of small complete graphs,
{\em J. Statist. Plann. Inference} {\bf  86} (2000), 435--442.

\bibitem{GloJoo} S. Glock, F. Joos, J. Kim, D. K\"{u}hn, D. Osthus,
Resolution of the Oberwolfach problem, {\em J. Eur. Math. Soc.} {\bf 23} (2021), 2511--2547.

\bibitem{Gvo} P. Gvozdjak,
On the Oberwolfach problem for cycles with multiple lengths,
PhD Thesis, Simon Fraser University, 2004.

\bibitem{Hag} R. H\"{a}ggkvist, A lemma on cycle decompositions, {\em Ann. Discrete Math.} {\bf  27} (1985), 227--232.

\bibitem{HofSch} D. G. Hoffman, P. J. Schellenberg, The existence of $C_k$-factorizations of $K_{2n}-F$, {\em Discrete Math.} {\bf 97} (1991),  243--250.

\bibitem{HuaKot} C. Huang, A. Kotzig, A. Rosa, On a variation of the Oberwolfach problem, {\em Discrete Math.} {\bf 27} (1979), 261--277.

\bibitem{Lac} A. Lacaze-Masmonteil, Completing the solution of the directed Oberwolfach problem with cycles of equal length, {\em J. Comb. Des.}, accepted Sep. 2023, arXiv:2212.12072.

\bibitem{Lac2} A. Lacaze-Masmonteil, private communication, Aug. 2023.



\bibitem{Mes} M. Meszka,
Solutions to the Oberwolfach problem for orders up to 100,
{\em Australas. J. Combin.} {\bf 89} (2024), 243--248.





\bibitem{SalDra} F. Salassa, G. Dragotto, T. Traetta, M. Buratti, F. Della Croce,  Merging combinatorial design and optimization: the Oberwolfach problem, {\em Australas. J. Combin.} {\bf 79} (2021), 141--166.

\bibitem{ShaSaj} E. Shabani, M. \v{S}ajna, On the Directed Oberwolfach Problem with variable cycle lengths, unpublished, arXiv:2009.08731 (2020).

\bibitem{Til} T. W. Tillson, A hamiltonian decomposition of $K_{2m}^\ast$, $2m \ge 8$, {\em J. Comb. Theory Ser. B} {\bf 29} (1980), 69--74.

\bibitem{Tra} T. Traetta,
A complete solution to the two-table Oberwolfach problems,
{\em J. Combin. Theory Ser. A} {\bf 120} (2013), 984--997.


\bibitem{Wes} Erik E. Westlund,
Hamilton decompositions of certain 6-regular Cayley graphs on Abelian groups with a cyclic subgroup of index two, {\em
Discrete Math.} {\bf 312} (2012), 3228--3235.
\end{thebibliography}


\normalsize
\end{document}